\documentclass{amsart}

\usepackage{amsthm}
\usepackage{graphicx}
\usepackage{enumerate}
\usepackage{amsmath,amssymb,amscd,amsfonts}
\usepackage{mathtools}
\usepackage{color}
\usepackage{lineno}

\allowdisplaybreaks[4]

\newtheorem{theorem}{Theorem}[section]
\newtheorem{lemma}[theorem]{Lemma}
\newtheorem{proposition}[theorem]{Proposition}

\theoremstyle{definition}

\theoremstyle{remark}
\newtheorem{remark}[theorem]{Remark}

\numberwithin{equation}{section}

\newcommand{\N}{\mathbb{N}}

\newcommand{\R}{\mathbb{R}}
\newcommand{\C}{\mathbb{C}}

\newcommand{\cB}{\mathcal{B}}
\newcommand{\cC}{\mathcal{C}}

\newcommand{\cF}{\mathcal{F}}

\newcommand{\cL}{\mathcal{L}}
\newcommand{\cN}{\mathcal{N}}
\newcommand{\cP}{\mathcal{P}}
\newcommand{\cS}{\mathcal{S}}
\newcommand{\cU}{\mathcal{U}}

\newcommand{\al}{\alpha}
\newcommand{\be}{\beta}
\newcommand{\la}{\lambda}
\newcommand{\La}{\Lambda}

\newcommand{\var}{\operatorname{var}}

\newcommand{\ua}{\underline{a}}
\newcommand{\Ua}{(a_i)_{i=1}^\infty}
\newcommand{\ub}{\underline{b}}

\newcommand{\ux}{\underline{x}}
\newcommand{\uy}{\underline{y}}

\newcommand{\uw}{\underline{w}}
\newcommand{\ud}{\underline{1}}
\newcommand{\um}{\underline{0}}

\begin{document}

\title[Dimension spectrum for beta-maps]{Dimension spectrum of digit frequency sets for beta-expansions}


\author{Shintaro Suzuki}
\address{Department of Mathematics, 
Tokyo Gakugei University, 
4-1-1 Nukuikitamachi 
Koganei-shi, 
Tokyo 184-8501,
Japan}
\email{shin05@u-gakugei.ac.jp}


\begin{abstract}
For any beta-shift $(X_\be,\sigma)$ on two symbols, i.e., the symbolic coding of the beta-map for $1<\be\leq2$, we give an exact formula for the Hausdorff dimension $\dim_{H} \La_{\al(t)}$ as a function of $t\in\R$, 
where $\La_\al$ denotes the frequency set of the digit $1$ defined by 
\[\La_\al=\Biggl\{(x_i)_{i=1}^\infty\in X_\be;\  \lim_{n\to\infty}\frac{1}{n}\sum_{i=1}^{n}x_i=\al \Biggr\}\]
for $\al\in[0,1]$ and $\al(t)$ is an explicit function related to the quasi-greedy expansion of $1$. The formula is derived from explicit formulae for eigenfunctions and eigenfunctionals corresponding to the leading eigenvalue $\la_t$ of the transfer operator $\cL_t$ with the potential $t\chi_{C_1}$ for $t\in\R$, where $\chi_{C_{1}}$ denotes the indicator function of the cylinder set $C_1=\{(x_i)_{i=1}^\infty\in X_\be; x_1=1\}$. These formulae can be applied not only to the leading eigenvalue but also to the other isolated eigenvalues of $\cL_t$, which yields a precise spectral decomposition of $\cL_t$. 
As a further application, we investigate the distribution function of the push-forward of the eigenmeasure  corresponding to $\la_t$ by the inverse map of the coding map. We show that the distribution function after a  change of variables for $t$ is equal to the Lebesgue singular function if $\be=2$ and satisfies an analogy of the Hata-Yamaguti formula, which yields a generalization of the Takagi function for beta-expansions with the base $1<\be<2$.

\end{abstract}
\maketitle
\begin{section}{Introduction}

The dynamical systems associated to expansions of real numbers have been investigated intensively after \cite{Re} at the intersection of ergodic theory and number theory. One of such dynamical systems is the beta-map $\tau_\be(x)=\be x$ mod $1$ for $x\in[0,1]$, $\be>1$, which equals to the dyadic map if $\be=2$. The beta-map is a simple example of a so-called chaotic map and has a unique absolutely continuous invariant probability measure $m_\be$, which is equivalent to the Lebesgue measure \cite{Re}. The ergodic properties of $m_\be$ are well-understood: its ergodicity \cite{Re}, an explicit formula for its density function \cite{Ge, Pa}, its exactness and the Bernoulli property of the natural extension of $\tau_\be$ \cite{Bo,Ro,Sm}. 
From these ergodic properties, we can investigate the statistical properties of the (greedy) beta-expansion: 
\[x=\sum_{n=1}^\infty\frac{g_n(x)}{\be^n}\]
for $x\in[0,1]$, where $g_n(x)$ denotes the integer part of $\be\tau_\be^{n-1}(x)$ for $x\in[0,1]$, $n\geq1$. 

One of our main concerns in this paper is the Hausdorff dimension of the frequency set   $\hat\La_\al$ of the digit $1$ appearing in beta-expansions, where the set $\hat\La_\al$ is defined by 
\[\hat\La_\al=\Biggl\{x\in[0,1]; \lim_{n\to\infty}\frac{1}{n}\sum_{i=1}^ng_i(x)=\al\Biggr\}\]
for $\al\in[0,1]$. Since the invariant measure $m_\be$ is ergodic with respect to $\tau_\be$, 
we have that the frequency set satisfies the $0$-$1$ law in the sense of the Lebesgue measure. In fact, 
the Birkhoff ergodic theorem yields that the relative frequency $(1/n)\sum_{i=1}^ng_i(x)$ converges to $\int_0^1g_1\ dm_\be$ as $n\to\infty$ for Lebesgue almost every point $x\in[0,1]$, together with that $m_\be$ is equivalent to the Lebesgue measure. Then it has full measure in $[0,1]$ for $\alpha=\int_0^1 g_1\ d m_\be$ but has zero measure for $\alpha\neq\int_0^1 g_1\ dm_\be$, which naturally poses the question of the value of its Hausdorff dimension. 
In the case of $\be=2$, it is well-known that the Hausdorff dimension $\dim_H \hat\Lambda_\al$ is given by 
\[\dim_H \hat\Lambda_\al=\frac{-\al\log\al-(1-\al)\log(1-\al)}{\log 2}\]
for $\al\in[0,1]$, which was shown by Besicovitch \cite{Be}.  For the golden ratio $\be=(1+\sqrt{5})/2$, Fan and Zhu \cite{Fa-Zh} gave an explicit formula:
\[\dim_{H}(\hat\Lambda_\al)=\frac{1}{\log\be}\Biggl((1-\al)\log\frac{1-\al}{1-2\al}-\al\log\frac{\al}{1-2\al}\Biggr)\]
for $\alpha\in[0,1/2)$ and $\dim_H\La_\al=0$ for $\al\in[1/2,1]$. 
Recently, Li \cite{Li} refined a conditional variational principle for the Hausdorff dimension (see also \cite{Ch-Zh-Zh, Th}) in the case where $\be\in(1,2)$ is a simple Parry number, i.e., the orbit of $1$ by $\tau_\be$ eventually falls in $\{0\}$ and gave its algebraic formula in the case where $\be$ is the tribonacci number, i.e., $\be>1$ satisfies $\be^3-\be^2-\be-1=0$ (see \cite[Remark 1]{Li}). From these results, it is natural to ask whether an explicit formula for $\dim_H\hat \La_\al$ can be given in general cases $1<\be<2$, although it would be more involved as a function of $\al\in[0,1]$, which is indicated for the case of the tribonacci number in \cite[Remark 1]{Li}.

In this paper, we give the Hausdorff dimension of the frequency set as a function of new parameter $t\in\R$ corresponding to the inverse temperature in the context of thermodynamic formalism for symbolic dynamics.
Furthermore, we show some related results brought by transfer operator method, including precise formulae for eigenfunctions and eigenfunctionals corresponding to an isolated eigenvalue of the transfer operator with a simple locally constant function. Throughout the paper, 
we consider the beta-shift, i.e., the coding system $(X_\be,\sigma)$ of the map $\tau_\be$ for $1<\be\leq2$, which allows us to investigate the beta-map as a continuous map on a compact metric space of two symbols without changing the essential properties of the beta-map. For the beta-shift the frequency set is defined by 
\[\La_\al=\Biggl\{(x_i)_{i=1}^\infty\in X_\be\subset\{0,1\}^\N;\  \lim_{n\to\infty}\frac{1}{n}\sum_{i=1}^{n}x_i=\al \Biggr\}\]
for $\al\in[0,1]$, whose Hausdorff dimension is equal to that of $\hat\La_\al$ for $\al\in[0,1]$ (see \cite[Theorem 1]{Li}). We give an exact formula for $\dim_H\La_{\al(t)}$ in Theorem \ref{Main D}, where $t\in\R$ and $\al(t)$ is a certain function defined using the coefficient sequence of the quasi-greedy expansion of $1$ (see Section 2 for the definition of quasi-greedy expansions). Our strategy of the proof is to use the conditional variational principle for the the Hausdorff dimension of the frequency set: 
\[\dim_{H}\La_{\al(t)}=\frac{h_{\mu_t}(\sigma)}{\log\be}\]
for $t\in\R$, where $h_{\mu_t}(\sigma)$ denotes the measure-theoretic entropy of a unique equilibrium state $\mu_t$ for the potential $t\chi_{C_1}$ and $\chi_{C_1}$ denotes the indicator function of $C_1=\{(x_i)_{i=1}^\infty\in X_\be; x_1=1\}$ (Lemma \ref{4-4}). Following the thermodynamic formalism for symbolic dynamics, the equilibrium state can be given by $\mu_t=h_t\nu_t$, where $h_t$ and $\nu_t$ are an eigenfunction and an eigenmeasure corresponding to the leading eigenvalue of the transfer operator $\cL_t$ with the potential $t\chi_{C_1}$, respectively (Proposition \ref{4-2}). In Theorem \ref{quasi}, we show that $\cL_t$ is quasi-compact, i.e., there is a positive number $r_t$ less than its spectral radius such that any spectral value whose modulus is greater than $r_t$ is an isolated eigenvalue with finite multiplicity. Furthermore, for an isolated eigenvalue $\la\in\C$ with $|\la|>r_t$, we give explicit formulae for an eigenfunction and an eigenfunctional in Theorems \ref{Main B} and \ref{Main A}, respectively. From these formulae, we construct an analytic function $\Phi_t(z)$ for $t\in\R$, which is regarded as a sort of the determinant of $\cL_t$ (see e.g., \cite{Ba-Ke, Fl-La-Po, Ho-Ke}) such that for $\la\in\C$ with $|\la|>r_t$ it is an isolated eigenvalue of $\cL_t$ if and only if its inverse is a zero of $\Phi_t$ (Theorem \ref{determinant}). This yields that there is a unique leading eigenvalue $\la_t>1$, which is simple (Theorem \ref{Main C} (2) and (3)) and gives an exact formula for the entropy $h_{\mu_t}(\sigma)$, which also yields that for $\dim_H\La_{\al(t)}$. As a simple application, we give an explicit formula for $\dim_H\La_{\al}$ as a function of $\al$ in the case where $\be>1$ satisfies $\be^{N+1}-\be^N-1=0$ for some $N\geq1$ (Theorem \ref{7-1}). 
As another applications of our explicit formulae, we investigate some additional properties of the equilibrium state $\mu_t$ and the pressure function $P(t)=\log \la_t$, where $\la_t$ is the leading eigenvalue of $\cL_t$ for $t\in\R$. In particular, we show the analyticity of $P(t)$ with an explicit formula for its derivative (Proposition \ref{5-1} (3)) and the exponential decay of correlation functions for $\mu_t$
with an optimal bound of the exponential rate derived from the other zeros of $\Phi_t$ (Theorem \ref{Main C} (3)). 
In addition, we give upper and lower bounds for the maximal value $c_\be$ with $\dim_H\La_\al>0$ for $\al<c_\be$ (Theorem \ref{5-2}). We note that some results in Section 5, like the analyticity of $P(t)$, have been shown in more general setting (see e.g., \cite{Cl}). We find that such a result is also proved directly in our setting as a consequence of the explicit formulae related to $\cL_t$ for $t\in\R$. 

As a further application of an explicit formula for an eigenmeasure $\nu_t$ corresponding to $\la_t$, we investigate the distribution function $D_t(x)=\nu_t(\pi_\be^{-1}([0,x]))$ for $x\in[0,1]$ and $t\in\R$, where $\pi_\be(\ua)=\sum_{n=1}^\infty a_n/\be^n$ for $\ua=(a_n)_{n=1}^\infty\in X_\be$. We define the function $F_p(x)=D_{f(p)}(x)$ for $x\in[0,1]$, where $f(p)$ is the inverse function of $\la_t^{-1}$ for $t\in\R$ (Proposition \ref{5-1} (1) and (2) ensure the existence of the inverse function of $\la_t^{-1}$). In Section 6 we see that $F_p(x)$ is equal to the Lebesgue singular function if $\be=2$ (see \cite{Al-Ka, Ha-Ya}).  Furthermore, in Theorem \ref{Main E}, we show that the function $F_p(x)$ is real-analytic at $p=1/\be$ for any $x\in[0,1]$ and $G_\be(x):=(1/\be)\partial F_p(x)/\partial p|_{p=1/\be}$ is continuous but nowhere differentiable on $[0,1]$. We note that the Hata-Yamaguti formula \cite[Theorem 4.6]{Ha-Ya} states that $G_2(x)$ is equal to the Takagi function, which is well-known as a continuous but nowhere differentiable function on $[0,1]$. Theorem \ref{Main E} shows that the function $G_\be(x)$ is a generalization of the Takagi function for beta-expansions with the base $1<\be<2$, whose fractal properties can be expected to relate to the algebraic properties of $\be$.  

This paper is organized as follows. 
In Section 2, we summarize some notions including beta-shifts and some basic properties of transfer operators. 
In Section 3, we give explicit formulae for an eigenfunction and eigenfunctional for an isolated eigenvalue of the transfer operator $\cL_t$ with the potential $t\chi_{C_1}$ for any $t\in\R$ and 
construct the analytic function $\Phi_t$ for $t\in\R$. 
In Section 4, we give an exact formula for $\dim_H\La_{\al(t)}$ as an application of the explicit formulae. 
In Section 5, we investigate some properties of the pressure function $P(t)=\log\la_t$ for $t\in\R$. In addition, we give upper and lower bounds for the maximal value $c_\be$ such that $\dim_H\La_\al>0$ for $\al<c_\be$. 
In Section 6, we show an analogy of the Hata-Yamaguti formula coming from an explicit formula for the eigenmeasure $\nu_t$ for $t\in\R$. 
In Section 7, we give an explicit formula for $\dim_H\La_\al$ as a function of $\al$ if $\be>1$ satisfies $\be^{N+1}-\be^N-1=0$ for some integer $N\geq1$.
\end{section}

\begin{section}{Preliminaries}

\begin{subsection}{Beta-maps and beta-shifts}
In this section we recall some necessary notions for beta-maps and beta-shifts following \cite{Bl, It-Ta, Pa, Re}. 
For $1<\be\leq2$, the beta-map is defined by
\[\tau(x):=\tau_\be(x)=\be x-[\be x]\]
for $[0,1]$, where $[y]$ denotes the integer part of $y\geq0$. It is well-known that the map $\tau$ gives the \textit{greedy expansion} of $x\in[0,1]$:
\[x=\sum_{n=1}^\infty\frac{g_n(x)}{\be^n},\]
where $g_n(x):=g_n(\be,x)=[\be\tau_{}^{n-1}(x)]$ for $n\geq1$. Note that the expansion is equal to the binary expansion in the case of $\be=2$. 

A real number $x\in[0,1]$ is said to be \textit{simple} if there is a positive integer $n_0\geq1$ such that $\tau^{n_0-1}(x)=1/\be$. In this case, we have $g_n(x)=[\be\tau_{}^{n-1}(x)]=0$ for all $n\geq n_0+1$. 
We set $S(x)=n_0$ if $x$ is simple and $S(x)=\infty$ otherwise.
$\be$ is said to be simple if $1$ is simple.

In some situations, it is more useful for us to consider the \textit{quasi-greedy expansion} of $x\in(0,1]$ defined as follows. 
If $\be$ is simple, we set  
\[(q_n(1))_{n=1}^\infty=\overline{g_1(1)\dots g_{S(1)-1}(1)({g}_{S(1)}(1)-1)}^{\infty},\]
where $\overline{b_1\dots b_k}^\infty$ denotes the infinite concatenation of a $k$-length word $b_1\dots b_k$ for $k\geq1$ with non-negative integers $b_1$,$\dots$, $b_k$.
If $\be$ is non-simple, we set $(q_n(1))_{n=1}^\infty=(g_n(1))_{n=1}^\infty$. By the definition of the coefficient sequence $(q_n(1))_{n=1}^\infty$ we have that $\displaystyle{1=\sum_{n=1}^\infty}q_{n}(1)/\be^n$. The right-hand side of this equality is called the quasi-greedy expansion of $1$.
For a simple number $x\in(0,1)$, we define
$q_i(x)=g_i(x)\text{ for } 1\leq i<S(x),\ q_{S(x)}=0\text{ and } q_{S(x)+i}(x)=q_i(1)\text{ for } i\geq1$.
By setting $(q_n(x))_{n=1}^\infty=(g_n(x))_{n=1}^\infty$ for a non-simple number $x\in(0,1)$,
we have that $\displaystyle{x=\sum_{n=1}^\infty}q_{n}(x)/\be^n$. The right-hand side of this equality is called the quasi-greedy expansion of $x\in(0,1)$.

We define the {\textit{beta-shift}} for $1<\be\leq2$. Denote by $D:=D_\be$ the set of all simple numbers. 
The coding map 
$i: [0,1]\setminus D\to \{0,1\}^\N$ on the set of non-simple numbers is defined by
$i(x)=(g_n(x))_{n=1}^\infty$ for $x\in [0,1]\setminus D$. Let $\sigma:\{0,1\}^\N\to\{0,1\}^\N$ be the {\textit{left shift}}, i.e., the map defined by $\sigma((\omega_i)_{i=1}^\infty)=(\omega_{i+1})_{i=1}^\infty$ for $(\omega_i)_{i=1}^\infty\in\{0,1\}^\N$. By definition we have $i\circ\tau=\sigma\circ i$, which ensures that $\sigma(i([0,1]\setminus D))\subset i([0,1]\setminus D)$. Set $X_\be=\overline{i([0,1]\setminus D)}$, where $\overline{A}$ denotes the closure of $A\subset\{0,1\}^\N$ in the sense of the product topology on $\{0,1\}^\N$. Since the map $\sigma$ is continuous on $\{0,1\}^\N$, we have $\sigma(X_\be)\subset X_\be$ and the restriction of $\sigma$ to $X_\be$ is also continuous on $X_\be$. Note that $X_\be$ is compact on $\{0,1\}^\N$. We call the topological dynamics $(X_\be,\sigma)$ the beta shift for $\be$.

Let $n\geq1$ and $(a_i)_{i=1}^n\in\{0,1\}^n$. We define the $n$-cylinder set for $(a_i)_{i=1}^n$ by 
\[[a_1\dots a_n]_1^n=\{(x_i)_{i=1}^\infty\in X_\be; x_i=a_i \text{ for } 1\leq i\leq n\}.\]
We denote by $\cC_n$ the set of all $n$-cylinder sets and by $\cS_n$ the set of all $n$-cylinder sets each of whose image by $\sigma$ is a whole space $X_\be$. We call an element in $\cS_n$ an $n$-full cylinder set. 

We equip the set $\{0,1\}^\N$ with the lexicographic order $\prec$, which is defined by 
$(s_n)_{n=1}^\infty\prec(t_n)_{n=1}^\infty$ if there is a positive integer $m\geq1$ such that $s_n=t_n$ for $1\leq n<m$ and $s_m<t_m$ for $(s_n)_{n=1}^\infty, (t_n)_{n=1}^\infty\in \{0,1\}^\N$. 
We denote by $(s_n)_{n=1}^\infty\preceq(t_n)_{n=1}^\infty$ 
if $(s_n)_{n=1}^\infty\prec(t_n)_{n=1}^\infty$ or $(s_n)_{n=1}^\infty=(t_n)_{n=1}^\infty$ Note that this is a total order and its order topology 
equals to the product topology on $\{0,1\}^\N$. For $\ua, \ub \in\{0,1\}^\N$ we define $[\ua, \ub]=\{\ux\in X_\be; \ua\preceq \ux\preceq\ub\}$. 

Since 
the coding map $i$ preserves the order relation, for $x\in(0,1]$
the limit $\lim_{y\nearrow x}(g_n(y))_{n=1}^\infty$ exists in the sense of the product topology in $X_\be$. We can see that $(q_n(x))_{n=1}^\infty=\lim_{y\nearrow x}(g_n(y))_{n=1}^\infty$ by the definition of the quasi-greedy expansion of $x\in(0,1]$. 
Throughout this paper, we write $\ud=(q_n)_{n=1}^\infty:=(q_n(1))_{n=1}^\infty$ and $\um=\overline{0}^\infty$ for simplicity. Clearly, $\um,\ud\in X_\be$. 
The next proposition is due to Parry \cite{Pa} (see also \cite{Bl}), which characterizes an element of $X_\be$ in terms of the order relation to the coefficient sequence $\ud$ of the quasi-greedy expansion of $1$. 

\begin{proposition}\label{Parry1}
For $1<\be\leq2$, we have 
\[X_\be=\{\ux\in\{0,1\}^\N;\ \sigma^n(\ux)\preceq \ud\ \text{ for } n\geq0\}.\]
\end{proposition}

\end{subsection}

\begin{subsection}{Transfer operators}
This subsection is devoted to summarizing the basic properties of a transfer operator for a beta-map and its spectral decomposition (see e.g., \cite{Ba, Ba-Ke, Bo-Go, Ho-Ke, La-Yo, Li-So-Va}). 
For $1<\be\leq2$ let $(X_\be,\sigma)$ be the corresponding beta-shift. For a function $f:X_\be\to\C$, we define the {\textit{total variation}} $\var(f)$ by
\begin{align*}
\var(f)=\sup\Biggl\{&\sum_{i=1}^n|f(x_i)-f(x_{i-1})|\ ;\ n\geq1, x_0=\um, x_n=\ud \\
&\text{ and }x_i\in X_\be \text{ with } x_{i-1}\prec x_i \text{ for } 1\leq i\leq n\Biggr\}.
\end{align*}
Let us denote by $BV$ the set of functions of bounded variation, i.e., 
\[BV=\{f:X_\be\to\C\ ;\ \var(f)<\infty\}.\]
We endow the space $BV$ with the norm 
\[||f||_{BV}=|f|_\infty +\var(f),\]
where $|f|_\infty=\sup_{x\in X_\be}|f(x)|$. Then $(BV, ||\cdot||_{BV})$ is a Banach space (see e.g., \cite{Ba, Li-So-Va}).

For $\phi\in BV$, the {\textit{transfer operator}} $\cL_\phi: BV\to BV$ with the potential $\phi$ is defined by
\[\cL_\phi f(x)=\sum_{x=\sigma(y)}e^{\phi(y)}f(y)\] 
for $f\in BV$ and $x\in X_\be$. 
We note that $\cL_\phi f$ is given by a linear sum of the product of a bounded function and a function of bounded variation, which yields
$\cL_\phi(BV)\subset BV$. 

In terms of application, it is more useful for us to introduce a quotient space of $BV$ on which functions are identified if they are equal except on a countable set. Let 
\[\cN=\{f\in BV\ ;\exists\text{ a countable set }N \text{ s.t. } f(x)=0 \text{ for } x\in X_\be\setminus N\}.\]
Note that the space $\cN$ is a closed linear subspace in $BV$.
We define the quotient Banach space $\cB=BV/\cN$ with the complete norm 
$||f||_\cB:=\inf_{g\in\cN}||f+g||_{BV}$ for $f\in \cB$.
Then the transfer operator $\cL_\phi$ is also linear and bounded on $\cB$ (see e.g., \cite{Ba, Li-So-Va}). 

We now introduce a closed linear subspace $\cF\subset \cB$ satisfying $\cL_\phi\cF\subset \cF$ for $\phi\in\cF$, which plays an important role in this paper. We denote by $\chi_A$ the indicator function of a Borel set $A\subset X_\be$. Let $\cU$ be the linear subspace of functions each of which is represented by a finite linear sum of the indicator functions of intervals in $X_\be$. Set $\cF=\overline{\cU}\subset \cB$, where the closure is taken in the sense of the norm topology in $(\cB, ||\cdot||_\cB)$. Then $(\cF, ||\cdot||_\cB)$ is a Banach space. For a real-valued function $\phi\in\cF$, since $e^\phi\in\cF$ and $e^\phi\geq0$, we know that $\cL_\phi\cF\subset \cF$ and $\cL_\phi f\geq0$ whenever $f\geq0$. Furthermore, $\cL_\phi$ is bounded and linear on $\cB$ and so is on $\cF$. Let us denote by $\cF^*$ the dual space of $\cF$, i.e.,the space of bounded linear complex-valued functionals on $\cF$. We define the {\textit{dual operator}} $\cL_\phi^*$ of the transfer operator $\cL_\phi$ by
\[\cL_\phi^*\nu=\nu\circ\cL_\phi\]
for $\nu\in\cF^*$. Note that $\cL^*_\phi: \cF^*\to \cF^*$ is also linear and bounded on $\cF^*$ with $||\cL^*_\phi||=||\cL_\phi||$, where $||\cdot||$ denotes the operator norm. 

For an eigenvalue $\la\in\C$ of $\cL_\phi$, 
the {\textit{geometric multiplicity}} of $\la$ is the dimension of the eigenspace $\{f\in \cF\ ;\ (\la I-\cL_\phi)f=0\}$, where $I$ denotes the identity map. The {\textit{(algebraic) multiplicity}} of $\la$ is the dimension of the generalized eigenspace $\{f\in \cF\ ; \exists n\geq 1 \text{ s.t. }  (\la I-\cL_\phi)^n f=0\}$.

\begin{proposition}\label{functional}
    For a real-valued function $\phi\in\cF$ there are a positive eigenvalue $\la_\phi$ of $\cL_\phi$ and an eigenfunctional $\nu_\phi\in\cF^*$ corresponding to $\la_\phi$  satisfying $\nu_\phi(f)\geq0$ for $f\geq 0$ and $\nu_\phi(1)=1$. 
    \end{proposition}

    \begin{proof}
We follow a basic argument using the Schauder fixed point theorem.
Let 
\[\Lambda=\bigl\{\nu\in\cF^*:\ \nu(f)\geq0\ \text{if}\ f\in\cF\ \text{ with } f\geq0\ \text{ and }\ \nu(1)=1\bigr\}.\]
We define the functional $\nu'$ by
\[\nu'(f)=\int_{X_\be} f dl'\]
for $f\in\cF$, where $l'=l\circ\pi_\be$ is a probability measure on $X_\be$ given by the Lebesgue measure $l$ on $[0,1]$ and the map $\pi_\be: X_\be\to [0,1]$ defined by $\pi_\be((a_n)_{n=1}^\infty)=\sum_{n=1}^\infty a_n/\be^n$ for $(a_n)_{n=1}^\infty\in X_\be$. Then it is easy to see that $\nu'\in\Lambda$, which ensures that $\Lambda\neq\emptyset$. By the definition of $\Lambda$, we know that it is a convex set which is closed in the sense of the weak-$^*$topology on $\cF^*$. Let $T:\Lambda\to \cF^*$ be the map defined by 
\[T(\nu)=\frac{\cL_\phi^*\nu}{(\cL_\phi^*\nu)(1)}\]
for $\nu\in\Lambda$. Since $(T\nu)(1)=1$
and $T\nu(f)\geq0$ whenever $f\geq0$, we have $T(\Lambda)\subset\Lambda$. Note that 
\[|(\cL_\phi^*\nu)(f)|=|\nu(\cL_\phi f)|\leq |\cL_\phi f|_\infty\nu(1)\leq |\cL_\phi1|_\infty|f|_\infty\]
for $\nu\in\Lambda$ and $f\in\cF$ with $f\geq0$. Together with the fact that 
\[(\cL_\phi^*\nu)(1)\geq e^{\phi(0\ux)}\geq e^{-|\phi|_\infty}\]
for $\nu\in\Lambda$, we have
\[T\nu(f)=\frac{(\cL_\phi^*\nu)(f)}{(\cL_\phi^*\nu)(1)}\leq\frac{|\cL_\phi1|_\infty}{e^{-|\phi|_\infty}}|f|_\infty\leq\frac{|\cL_\phi1|_\infty}{e^{-|\phi|_\infty}}||f||_\cB,\]
which shows that 
\[||T\nu||\leq \frac{|\cL_\phi1|_\infty}{e^{-|\phi|_\infty}}\]
for $\nu\in\Lambda$. Hence $T\Lambda$ is bounded. This yields that $T\Lambda$ is included in some closed ball in $\cF^*$. By the Schauder fixed point theorem, there is a fixed point $\nu_\phi$ of $T$ on $\Lambda$. This yields that 
\[\frac{\cL_\phi^*\nu_\phi}{(\cL_\phi^*\nu_\phi)(1)}=\nu_\phi.\]
By setting $\la_\phi=(\cL_\phi^*\nu_\phi)(1)>0$, we have that 
\[\cL_\phi^*\nu_\phi=\la_\phi\nu_\phi\]
and $\nu_\phi(1)=(\cL_\phi^*\nu_\phi)(1)/(\cL_\phi^*\nu_\phi)(1)=1$, 
as desired.

\end{proof}

For a real-valued function $f$ and $n\geq1$ we write $S_nf=\sum_{i=0}^{n-1}f\circ\sigma^i$. For a function $\phi\in\cF$ we set
\[r_\phi=\limsup_{n\to\infty}{\sqrt[n]{\exp\Bigl(\sup_{\ux\in X_\be}S_n\phi(\ux)}\Bigr)}=\limsup_{n\to\infty}\sqrt[n]{|e^{S_n\phi}|_\infty}.\]
In the following, we show the Lasota-Yorke type inequality for $\cL_\phi$.

\begin{lemma}
    Let $\phi$ be a real-valued function in $\cF$ which is Lipschitz continuous on $X_\be$. Then 
    \begin{equation}\label{lyi}
        \var(\cL_\phi^n f)\leq |e^{S_n\phi}|_\infty\var(f)+(|e^{S_n\phi}|_\infty c_n+\var(e^{S_n\phi}))|f|_\infty,
    \end{equation}
    where $c_n$ denotes the number of the set $\cC_n$ of all $n$-cylinder sets. 
\end{lemma}

\begin{proof}
For $n\geq1$ let $\um=\ux_0\prec \ux_1\prec \ux_2\prec\dots\prec \ux_{c_n}=\ud$ be the endpoints of any elements in $\mathcal{C}_n$. For $\ua, \ub \in X_\be$ with $\ua\prec \ub$ and a function $f:X_\be\to\C$, we denote by $\var_{\ua}^{\ub}(f)$ the total variation of $f$ on $[\ua, \ub]$. Then for any points $\um=\uw_0\prec \uw_1\prec \dots\prec \uw_{N}=\ud$ we have
\begin{align*}
    \sum_{i=1}^N|\cL_\phi^nf(\uw_i)-\cL_\phi^nf(\uw_{i-1})|
&=\sum_{i=1}^N\Biggl|\sum_{\uw_i=\sigma^n(y)}e^{S_n\phi(y)}f(y)-\sum_{\uw_{i-1}=\sigma^n(z)}e^{S_n\phi(z)}f(z)\Biggr| \\
&\leq \sum_{i=1}^{c_n}\bigl(\var_{\ux_{_{i-1}}}^{\ux_i}(e^{S_n\phi} f)+|e^{S_n\phi} f|_\infty\bigr) \\
    &\leq \var(e^{S_n\phi}f)+c_n|e^{S_n\phi}f|_\infty \\
    &\leq |e^{S_n\phi}|_\infty\var(f)+\var(e^{S_n\phi})|f|_\infty+c_n|e^{S_n\phi}f|_\infty \\
    &\leq|e^{S_n\phi}|_\infty\var(f)+(|e^{S_n\phi}|_\infty c_n+\var(e^{S_n\phi}))|f|_\infty,
\end{align*}
which ends the proof. 
\end{proof}

Note that $\cL_\phi$ is a compact operator from $(\cF,||\cdot||_{\cB})$ to $(\cF,|\cdot|_{\infty})$, which follows from that a bounded set in the sense of the norm $||\cdot||_{\cB}$ is relatively compact in the topology from the norm $|\cdot|_{\infty}$(see e.g., \cite{Ba, Li-So-Va}). Since $\la_\phi$ is a positive eigenvalue, the spectral radius of $\cL_\phi$ is greater than or equal to $\la_\phi$. 

If there is a positive integer $n_0$ such that 
\[\alpha:=\frac{|e^{S_{n_0}\phi}|_\infty}{\la_{\phi}^{n_0}}<1,\] 
then for any $n=kn_0+r$, where $k\geq0$ and $0\leq r\leq n_0-1$, we have 
\[\frac{|e^{S_{n}\phi}|_\infty}{\la_{\phi}^{n}}
\leq \frac{|e^{S_{r}\phi}|_\infty}{\la_{\phi}^{r}}
\Biggl(\frac{|e^{S_{n_0}\phi}|_\infty}{\la_{\phi}^{n_0}}\Biggr)^k
\leq \max_{0\leq r \leq n_0-1}\Biggl(\frac{|e^{S_{r}\phi}|_\infty}{\la_{\phi}^{r}}\Biggr)\alpha^k,\]
which yields 
\[\sqrt[n]{\frac{|e^{S_n\phi}|_\infty}{\la_{\phi}^n}}
\leq\sqrt[n]{ \max_{0\leq r \leq n_0-1}\Biggl(\frac{|e^{S_{r}\phi}|_\infty}{\la_{\phi}^{r}}\Biggr)\alpha^k}\to \alpha^{\frac{1}{n_0}}<1\]
as $n\to\infty$. This shows that 
\[r_\phi=\limsup_{n\to\infty}\sqrt[n]{|e^{S_n\phi}|_\infty}\leq \alpha^{\frac{1}{n_0}}\la_\phi<\la_\phi.\]
In particular, if 
$|e^{S_n\phi}|_\infty/\la_\phi^n\to0$ as $n\to\infty$, we obtain that there is a positive integer $n_0$ such that $|e^{S_{n_0}\phi}|_\infty/\la_{\phi}^{n_0}<1$.  
Applying \cite[Th\'eor\`eme 1]{He} (see also \cite[Remark 1.2]{Ba}) to our setting, we have the quasi-compactness of $\cL_\phi$:
\begin{theorem}\label{quasi}
   Assume that $\phi\in \cF$ is a real-valued and Lipshitz continuous function with $|e^{S_n\phi}|_\infty/\la_\phi^n\to0$ as $n\to\infty$. 
   Then the transfer operator $\cL_\phi$ is quai-compact. That is, a spectral value $\la$ with $|\la|>r_\phi$ is an isolated eigenvalue with finite multiplicity. In particular, $\la_\phi$ is a positive isolated eigenvalue with finite multiplicity. 
\end{theorem}
Let us denote by $R(\cL_\phi)$ the spectral radius of $\cL_\phi$. Under the same assumptions as in Theorem \ref{quasi}, 
$\cL_\phi$ is quasi-compact on $\cF$, which yields that for $\theta\in(r_\phi, R(\cL_\phi))$ the spectral decomposition of $\cL_\phi$ (see e.g., \cite{Ha}) is given by
\begin{equation}\label{decomposition}
\cL_\phi f=\sum_{i=1}^N \la_i h_i^{\bot} J_{i} \nu_{i}(f) +\cP\cL_{\phi}
\end{equation} 
for $f\in\cF$, where $\la_i$ is an isolated eigenvalue with finite multiplicity $M_i \geq1$ with $|\la_i|>\theta$, $h_i$ is the vector of a basis $(h_{i,1},\dots, h_{i,M_i})$ for the corresponding generalized eigenspace, $J_i$ is a Jordan matrix composed by some Jordan blocks each of whose diagonals are $1$, and $\nu_i$ is the vector of eigenfunctionals $(\nu_{i,1},\dots,\nu_{i,M_i})$ satisfying $\nu_{i,j}(h_{k,l})=1$ if $i=k$ and $j=l$, and $\nu_{i,j}(h_{k,l})=0$ otherwise. 
Here $\cP$ is a linear operator whose spectral radius is less than or equal to $\theta$.

\end{subsection}

\end{section}

\begin{section}{Some explicit formulae related to transfer operators}
This section is devoted to giving some explicit formulae for an eigenfunction and an eigenfunctional corresponding to an isolated eigenvalue of $\cL_t:=\cL_{t\chi_{C_1}}$ for $t\in\R$, where $\chi_{C_1}$ denotes the indicator function of the cylinder set $C_1=\{(x_i)_{i=1}^\infty\in X_\be; x_1=1\}$.  
Since the potential $t\chi_{C_1}$ is a real-valued function in $\cF$, we can apply Proposition \ref{functional} to $\phi=\phi_t=t\chi_{C_1}$. In the following, we denote $\la_t=\la_{t\chi_{C_1}}$, $\nu_t=\nu_{t\chi_{C_1}}$ and $r_t=r_{t\chi_{C_1}}$ 
for simplicity. 
The following result states  that the transfer operator $\cL_t$ for any $t\in\R$ is quasi-compact, for which more general cases are treated in \cite{Li-So-Va, Wa2}. We show that the potential $t\chi_{C_1}$ satisfies the assumptions in Theorem \ref{quasi}
for any $t\in\R$ using the relation between the number of full $n$-cylinders and the value of $\inf_{\ux\in X_\be}(\cL_{t}^n1)(\ux)$ for $n\geq1$.

\begin{proposition}\label{hyperbolic potential}
    For any $t\in\R$ the potential function $t\chi_{C_1}$ satisfies the assumptions in Theorem \ref{quasi}. In particular, $r_t<\la_t$ and $\cL_t$ is quasi-compact on $\cF$. 
\end{proposition}

\begin{proof} 
For any $t\in\R$, the function $t\chi_{C_1}$ is locally constant on $C_1$ and $C_1^c$, which is also Lipshitz continuous. To see that $|e^{S_n\phi_t}|_\infty/\la_t^n\to0$ as $n\to\infty$, it is sufficient to show that $|e^{S_n\phi_t}|_\infty/\inf_{\ux\in X_\be} (\cL_t^n1)(\ux)\to0$ as $n\to\infty$  since 
    \[\la_t^n=\la_t^n\nu_t(1)=\nu_t(\cL_t^n1)\geq \inf_{\ux\in X_\be} (\cL_t^n1)(\ux),\]
which yields $|e^{S_n\phi_t}|_\infty/\la_t^n\leq |e^{S_n\phi_t}|_\infty/\inf_{\ux\in X_\be} (\cL_t^n1)(\ux)\to0$ as $n\to\infty$. 

For $n\geq1$ the function $\psi_n:=\sum_{i=0}^{n-1}\chi_{C_1}\circ \sigma^i$ is locally constant and actually constant on any $I\in \cC_n$. Then $\psi_n$ reaches its maximum in some $J\in \cC_n$. Let $n\geq1$ be so large as $|\psi_n|_\infty\geq3$. Take $\ux'=(x'_i)_{i=1}^\infty\in X_\be$ so that $\psi_n(\ux')=|\psi_n|_{\infty}$ and set $I'=[x_1'\dots x_n']_1^n$. By definition we know that $\sum_{i=1}^n x_i'=\psi_n(\ux')$. Set $N=\psi_n(\ux')$. Then there is a sequence of positive integers $(n_k)_{k=1}^N$ such that $x_{n_k}=1$ for $1\leq k\leq N$. For $1\leq i\leq N-1$ we set $J_i=[y_{i,1}\dots y_{i,n}]_1^n$ where $y_{i,m}=1$ if $m=n_k$ for $k\in\{1,\dots, N-1\}\setminus\{i\}$ and $y_{i,m}=0$ otherwise. Then each $J_i$ is a full $n$-cylinder set since 
    \[y_{i,1}\dots y_{i,n}q_1q_2\dots \in X_\be\]
    for $1\leq i\leq N-1$ by Proposition \ref{Parry1}, where $(q_n)_{n=1}^\infty=(q_n(1))_{n=1}^\infty$ is the coefficient sequence of the quasi-greedy expansion of $1$. By the definition of $J_i$ we know that $\sum_{m=1}^n y_{i,m}=N-2$ for $1\leq i\leq N-1$. This shows that 
    \begin{align*}
        \inf_{\ux\in X_\be} (\cL_t^n1)(\ux)&=(\cL_t^n1)(\ud)=\sum_{[a_1\dots a_n]_1^n\in \cS_n}\exp\Biggl(\sum_{i=1}^n a_i\cdot t\Biggr) \\
        &\geq (N-1)\exp((N-2)t).
    \end{align*}
    Since $N\to\infty$ as $n\to\infty$ we have 
    \begin{align*}
        \frac{|e^{S_n\phi_t}|_\infty}{\inf_{\ux\in X_\be} (\cL_t^n1)(\ux)}\leq \frac{\exp(Nt)}{(N-1)\exp((N-2)t)}=\frac{e^{2t}}{N-1}\to0
    \end{align*}
    as $n\to\infty$, which gives the conclusion. 
\end{proof}

The next equation is a key tool for the proofs of the main results. 

\begin{lemma}\label{key lemma}
    For $\ua=(a_i)_{i=1}^\infty\in X_\be$, we have 
\begin{equation}
\cL_t(\chi_{[\um,\ua]})=a_1\chi_{[\um,\ud]}+l(a_1)\chi_{[\um,\sigma(\ua)]},
\end{equation}
where $l(a_1)$ is given by
\[l(a_1)=\begin{cases}
1 & \text{ if } a_1=0, \\
e^t & \text{ if } a_1=1.
\end{cases}\]
\end{lemma}

\begin{proof}
    By Proposition \ref{Parry1}, we have 
    \[\sigma^{-1}(\{\ux\})=
    \begin{cases}
        \{0\ux, 1\ux \} & \text{ if } \ux\preceq\sigma(\ud), \\
        \{0\ux\} & \text{ if } \ux\succ\sigma(\ud) \\
    \end{cases}
    \]
    for $\ux\in X_\be$. 
    
Assume that $\ua=\Ua\in X_\be$ satisfies $a_1=0$. Then 
\[\begin{split}
    \cL_t\chi_{[\um,\ua]}(\ux)
    &=\sum_{\uy: \ux=\sigma(\uy)}e^{t\chi_{C_1}(\uy)}\chi_{[\um,\ua]}(\uy) \\
    &=\begin{cases}
        e^{t\chi_{C_1}(0\ux)}\chi_{[\um,\ua]}(0\ux)+e^{t\chi_{C_1}(1\ux)}\chi_{[\um,\ua]}(1\ux) & \text{ if } \ux\preceq\sigma(\ud), \\
        e^{t\chi_{C_1}(0\ux)}\chi_{[\um,\ua]}(0\ux)
        & \text{ if } \ux\succ\sigma(\ud)
    \end{cases} \\
    &=e^{t\chi_{C_1}(0\ux)}\chi_{[\um,\ua]}(0\ux) \\
    &=\begin{cases}
        1& \text{ if } \ux\preceq\sigma(\ua), \\
        0 & \text{ if } \ux\succ\sigma(\ua)
    \end{cases}\\
    &=0\cdot\chi_{[\um,\ud]}(\ux)+1\cdot\chi_{[\um,\sigma(\ua)]}(\ux),
\end{split}\]
which gives the conclusion. 

Assume that $\ua=\Ua\in X_\be$ satisfies $a_1=1$. Then 
\[\begin{split}
    \cL_t\chi_{[\um,\ua]}(\ux)
    &=\sum_{\uy: \ux=\sigma(\uy)}e^{t\chi_{C_1}(\uy)}\chi_{[\um,\ua]}(\uy) \\
    &=\begin{cases}
        e^{t\chi_{C_1}(0\ux)}\chi_{[\um,\ua]}(0\ux)+e^{t\chi_{C_1}(1\ux)}\chi_{[\um,\ua]}(1\ux) & \text{ if } \ux\preceq\sigma(\ud), \\
        e^{t\chi_{C_1}(0\ux)}\chi_{[\um,\ua]}(0\ux)
        & \text{ if } \ux\succ\sigma(\ud)
    \end{cases} \\
    &=1+\chi_{[0,\sigma(\ud)]}(\ux)e^{t\chi_{C_1}(1\ux)}\chi_{[\um,\ua]}(1\ux) \ \ \ (\ux\neq \sigma(\ud))\\
    &=\begin{cases}
        1+e^t& \text{ if } \ux\preceq\sigma(\ua), \\
        1 & \text{ if } \ux\succ\sigma(\ua)
    \end{cases}\\
    &=1\cdot\chi_{[\um,\ud]}(\ux)+e^t\cdot\chi_{[\um,\sigma(\ua)]}(\ux),
\end{split}\]
which ends the proof. 
\end{proof}

For an integer $n\geq0$ and $t\in\R$ we define the function $H_n(t,\cdot):X_\be\to\R$ by
\[H_n(t,\ux)=
\begin{cases}
    1 & \text{ if } n=0 \\
    \exp\Biggl(\sum_{i=1}^n x_i\cdot t\Biggr)& \text{ if } n\geq1
\end{cases}\]
for $\ux=(x_i)_{i=1}^\infty\in X_\be$.
As a generalization of \cite[Proposition 4.2]{Su2}, we give the form of an eigenfunctional corresponding to an isolated eigenvalue of $\cL_t$.

\begin{theorem}\label{Main A}
    Let $\la\in\C$ be an isolated eigenvalue of $\cL_t$ with $|\la|>r_t$. For every non-zero eigenfunctional $\nu_\la$ of $\la$ we have
    \[\nu_\la(\chi_{[\um,\ua]})=\nu_\la(\chi_{[\um,\ud]})\cdot\sum_{n=1}^\infty\frac{a_n H_{n-1}(t,\ua)}{\la^n}\]
    for $\ua=\Ua\in X_\be$. In addition, $\nu_\la(\chi_{[\um,\ud]})\neq0$. 
\end{theorem}
\begin{proof}
    Let $\ua=\Ua\in X_\be$. By Lemma \ref{key lemma} and the equation $\cL_t^*\nu_\la=\la\nu_\la$, we have

\[\begin{split}
    \nu_\la(\chi_{[\um,\ua]})
    &=\frac{\nu_\la(\cL_t\chi_{[\um,\ua]})}{\la} \\
    &=\frac{a_1}{\la}\nu_\la(\chi_{[\um,\ud]})+\frac{l(a_1)}{\la}\nu_\la(\chi_{[\um,\sigma(\ua)]}) \\
    &=\frac{a_1}{\la}\nu_\la(\chi_{[\um,\ud]})
    +\frac{a_2l(a_1)}{\la^2}\nu_\la(\chi_{[\um,\ud]})
    +\frac{l(a_1)l(a_2)}{\la^2}\nu_\la(\chi_{[\um,\sigma^2(\ua)]}) \\
    &=\dots \\
    &=\sum_{n=1}^N\frac{a_n H_{n-1}(t, \ua)}{\la^n}\nu_\la(\chi_{[\um,\ud]})+\frac{H_N(t,\ua)}{\la^N}\nu_\la(\chi_{[\um,\sigma^N(\ua)]})
\end{split}\]
for every positive integer $N$. Note that $|\nu_\la(\chi_{[0,\sigma^N(\ua)]})|\leq 2||\nu_\la||$ for any $N$, where $||\nu_\la||$ denotes the norm of $\nu_\la$ on the dual Banach space $\cF^*$. Since $H_{N}(t,\ua)/|\la^N|\leq |e^{S_N\phi_t}|_\infty/|\la|^N\to0$
as $N\to\infty$, we get the desired formula by taking $N\to\infty $ in the right side of the above equation. 
We note that $\cF$ is the closure of the set of all functions given by a finite linear sum of the indicator functions of the form $\chi_{[0,\ua]}$, where $\ua\in X_\be$. If $\nu_\la(\chi_{[0,\ud]})=0$ we have $\nu_\la(\chi_{[0,\ub]})=0$ for any $\ub\in X_\be$ by the above equation, which contradicts that $\nu_\la$ is a non-zero bounded linear functional on $\cF$. This gives the conclusion.  
\end{proof}

By Theorem \ref{Main A}, the eigenfunctional $\nu_t$ corresponding to $\la_t$ can be expressed as a non-atomic probability measure on $X_\be$. 

\begin{proposition}\label{3-4} Let $t\in\R$ and let $\nu_t$ be the eigenfunctional corresponding to $\la_t$. 
    (1) For an isolated eigenvalue $\la\in\C$ of $\cL_t$ with $|\la|>r_t$, we have
    \[\sum_{n=1}^\infty\frac{g_n(x)H_{n-1}(t, (g_m(x))_{m=1}^\infty)}{\la^n}=\sum_{n=1}^\infty\frac{q_n(x)H_{n-1}(t, (q_m(x))_{m=1}^\infty)}{\la^n}\]
    for $x\in(0,1]$.

    (2) We have 
    \[\nu_t([\um, \ua])=\lim_{\ux\to\ua}\nu_t([\um, \ux])\]
    for $\ua\in X_\be$.
    
    (3) There is a non-atomic measure $m_t$ on $X_\be$ such that 
    \[m_t([\um, \ua])=\nu_t(\chi_{[\um,\ua]})\]
    for $\ua\in X_\be$. 
\end{proposition}

\begin{proof}
    (1) Assume that $x\in(0,1]$ is non-simple. Then $(g_n(x))_{n=1}^\infty=(q_n(x))_{n=1}^\infty$, which gives the desired equality. Assume that $x\in(0,1]$ is simple, i.e., there is a positive integer $n_0$ such that $\tau_{\be}^{n_0-1}(x)=1/\be$. Then we know that $g_{n_0}(x)=1$ and $g_{n_0+i}(x)=0$ for $i\geq1$. By setting $\ua=\ud$ in Theorem \ref{Main A}, we have
   \[\nu_\la(\chi_{[\um,\ud]})=\nu_\la(\chi_{[\um,\ud]})\cdot
   \sum_{n=1}^\infty\frac{q_n(1) H_{n-1}(t,(q_m(1))_{m=1}^\infty)}{\la^n},\]
which yields 
\[\sum_{n=1}^\infty\frac{q_n(1) H_{n-1}(t,(q_m(1))_{m=1}^\infty)}{\la^n}=1.\]
By the definition of the quasi-greedy expansion $(q_m(x))_{m=1}^\infty$ of $x$, we obtain 
\[
\begin{split}
    &\sum_{n=1}^\infty\frac{q_n(x) H_{n-1}(t,(q_m(x))_{m=1}^\infty)}{\la^n} \\
    &=\sum_{n=1}^{n_0}\frac{g_n(x) H_{n-1}(t,(q_m(x))_{m=1}^\infty)}{\la^n}
    -\frac{H_{n_0-1}(t, (q_m(x))_{m=1}^\infty)}{\la^{n_0}} \\
    &+\sum_{n=n_0+1}^\infty\frac{q_n(x) H_{n-1}(t,(q_m(x))_{m=1}^\infty)}{\la^n}    \\
    &=\sum_{n=1}^{n_0}\frac{g_n(x) H_{n-1}(t,(g_m(x))_{m=1}^\infty)}{\la^n}
    -\frac{H_{n_0-1}(t, (q_m(x))_{m=1}^\infty)}{\la^{n_0}} \\
    &+\frac{H_{n_0-1}(t,(q_m(x))_{m=1}^\infty)}{\la^{n_0}}\sum_{n=1}^\infty\frac{q_n(1) H_{n-1}(t,(q_m(1))_{m=1}^\infty)}{\la^n}    \\
    &=\sum_{n=1}^{n_0}\frac{g_n(x) H_{n-1}(t,(g_m(x))_{m=1}^\infty)}{\la^n} \\
    &-\frac{H_{n_0-1}(t,(q_m(x))_{m=1}^\infty)}{\la^{n_0}}\Biggl(1-\sum_{n=1}^\infty\frac{q_n(1) H_{n-1}(t,(q_m(1))_{m=1}^\infty)}{\la^n}\Biggr) \\
    &=\sum_{n=1}^{n_0}\frac{g_n(x) H_{n-1}(t,(g_m(x))_{m=1}^\infty)}{\la^n},    
    \end{split}
\]
which gives the conclusion. 

(2) For $(a_i)_{i=1}^\infty\in X_\be$ set $\pi(\ua)=\sum_{n=1}^\infty a_n/\be^n$. By the definition of the quasi-greedy expansion of $x$ and that of the greedy expansion, we obtain 
\[\lim_{\ux\nearrow\ua}\nu_t([\um,\ux])=\sum_{n=1}^\infty\frac{q_n(\pi(\ua)) H_{n-1}(t,(q_m(\pi(\ua)))_{m=1}^\infty)}{\la^n}\]
and 
\[\lim_{\ux\searrow\ua}\nu_t([\um,\ux])=\sum_{n=1}^\infty\frac{g_n(\pi(\ua)) H_{n-1}(t,(g_m(\pi(\ua)))_{m=1}^\infty)}{\la^n},\]
\[\]
which gives the conclusion by Item (1). 

(3) For $x\in[0,1]$ we define 
\[\widehat m_t([0,x])=\nu_t([\um, (g_n(x))_{n=1}^\infty]).\]
Since the map $x\mapsto(g_n(x))_{n=1}^\infty$ is strictly increasing, together with the fact that $\nu_t(f)\leq \nu_t(g)$ for $f,g\in\cF$ with $f\leq g$, we have 
\[\widehat m_t([0,x])=\nu_t([\um,(g_n(x))_{n=1}^\infty])\leq \nu_t([\um,(g_n(y))_{n=1}^\infty])=\widehat m_t([0,y])\]
for $x,y\in[0,1]$ with $x\leq y$. Set $\widehat F_t(x)=\widehat m_t([0,x])$ for $x\in[0,1]$. Since $\widehat F_t(x)$ is continuous on $[0,1]$ by Item (2), together with the fact that $\widehat F_t(0)=0$ and $\widehat F_t(1)=1$, we have that it is the distribution function of some non-atomic probability measure $\widehat m_t'$ on $[0,1]$. By setting $m_t=\widehat m_t'\circ\pi$, we have 
\[\begin{split}
    m_t([\um, \ua])&=(\widehat m_t'\circ\pi)([\um,\ua])=\nu_t([\um,(g_n(\pi(\ua))_{n=1}^\infty]) \\
    &=\sum_{n=1}^\infty\frac{g_n(\pi(\ua))H_{n-1}(t,(g_m(\pi(\ua))_{m=1}^\infty)}{\la^n}\\
    &=\lim_{\ux\searrow\ua}\nu_t([\um,\ux])=\nu_t([\um,\ua]),
\end{split}\]
which gives the conclusion. 
\end{proof}

For $t\in\R$ we define the formal power series $\phi_t(z)$ by 
\begin{equation}\label{power series}
\phi_t(z)=\sum_{n=1}^\infty q_n(1) H_{n-1}(t,\ud)z^n,
\end{equation}
where $\ud=((q_n(1))_{n=1}^\infty$.
In the following, we show that the convergence radius of 
$\phi_t(z)$ is greater than or equal to $r_t^{-1}$ and an isolated eigenvalue of $\cL_t$ is the inverse of a zero of $1-\phi_t(z)$, which enables us to investigate isolated eigenvalues of $\cL_t$ via zeros of $1-\phi_t(z)$.

\begin{proposition}\label{phi_t}
    For $t\in\R$ let $\phi_t(z)$ be the power series as defined above. Then we have the following:

    (1) The convergence radius of $\phi_t(z)$ is greater than or equal to $r_t^{-1}$.

    (2) $1-\phi_t(z)$ has a unique positive zero $\eta_t$. 

    (3) If $\la$ is an isolated eigenvalue of $\cL_t$ with $|\la|>r_t$, then $\la^{-1}$ is a zero of $1-\phi_t(z)$. In particular, $\la_t^{-1}$ is a zero of $1-\phi_t(z)$ and $\la_t^{}=\eta_t^{-1}$.

    (4) Any zero $\eta$ of $1-\phi_t(z)$ except $\la_t^{-1}$ satisfies $\la_t^{-1}<|\eta|$.
    
\end{proposition}

\begin{proof}
(1) The direct calculation:
\[\limsup_{n\to\infty}\sqrt[n]{|q_n(1) H_{n-1}(t,\ud)|}=\limsup_{n\to\infty}\sqrt[n]{\exp\Biggl(\sum_{i=1}^{n-1}q_i(1)\cdot t\Biggr)}\leq \limsup_{n\to\infty}\sqrt[n]{|e^{S_n\phi_t}|_\infty}=r_t\]
gives the conclusion.

(2) Note that each coefficient $q_n(1) H_{n-1}(t,\ud)$ of $\phi_t$ is non-negative for $n\geq0$ and there are infinitely many $m$ such that $q_m(1) H_{m-1}(t,\ud)>0$ by the definition of $(q_n(1))_{n=1}^\infty$. Since $\phi_t(0)=0$
there is a unique positive number $\eta_t$ such that $\phi_t(\eta_t)=1$, which yields the conclusion. 

(3) Let $\la$ be an isolated eigenvalue of $\cL_t$ with $|\la|>r_t$. By Theorem \ref{Main A}, for an eigenfunctional $\nu$ of $\la$ we have
\[\nu_\la(\chi_{[\um,\ua]})=\nu_\la(\chi_{[\um,\ud]})\cdot\sum_{n=1}^\infty\frac{a_n H_{n-1}(t,\ua)}{\la^n}\]
for $\ua\in X_\be$. Setting $\ua=\ud$
yields 
\[\nu_\la(\chi_{[\um,\ud]})
\Biggl(1-\sum_{n=1}^\infty\frac{q_n(1) H_{n-1}(t,\ud)}{\la^n}\Biggr)=0.\]
Since $\nu(\chi_{[\um,\ud]})\neq0$, we obtain $1-\phi_t(\la^{-1})=0$. In particular, $\la_t^{-1}$ is a positive zero of $1-\phi_t(z)$ since $\la_t>r_t$ by Proposition \ref{hyperbolic potential}. From Item (2), $\la_t^{-1}$ is equal to $\eta_t$, which ends the proof. 

(4) Since any coefficient $q_n(1) H_{n-1}(t,\ud)$ of $\phi_t(z)$ is non-negative for $n\geq0$ and $\phi_t(r)$ is strictly increasing on $r\in[0,r_t^{-1})$, we have that $|\phi_t(z)|<1$ for $z\in\C$ with $|z|\leq \la^{-1}_t$ and $z\neq \la_t^{-1}$ (see e.g., \cite{Er-Fe-Po}). Together with Item (3), we obtain the conclusion. 
\end{proof}

Now we show that a complex number $\la\in\C$ satisfying $|\la|>r_t$  and $1-\phi_t(\la^{-1})=0$
is actually an isolated eigenvalue of $\cL_t$. This can be done by constructing an eigenfunction of $\la$ explicitly, which is a partial generalization of \cite[Theorem 3.2]{Su1}.

\begin{theorem}\label{Main B}
    Let $t\in\R$ and let $\la^{-1}\in\C$ be a zero of $1-\phi_t(z)$ with $\la_t^{-1}\leq |\la^{-1}|<r_t^{-1}$. Define the function $g_\la$ by
    \[g_\la=\sum_{n=0}^\infty\frac{H_n(t,\ud)}{\la^n}\chi_{[\um,\sigma^n(\ud)]}.\]
Then $g_\la$ is a non-zero function in $\cF$ satisfying $\cL_t g_\la=\la g_t$.
In particular, $\la$ is an eigenvalue of $\cL_t$.
\end{theorem}

\begin{proof}
Since
$\la_t^{-1}\leq|\la^{-1}|<r_t^{-1}$,
we know that $H_n(t,\ud)/|\la|^n\leq |e^{S_n\phi_t}|_\infty/|\la|^n$ converges to $0$ in exponentially fast as $n\to\infty$. Hence $g_\la$ is well-defined. In addition, since $\chi_{[\um,\sigma^n(\ud)]}(1\overline{0}^\infty)=q_{n+1}(1)$ for $n\geq0$, we have
\[g_\la(1\overline{0}^\infty)=\sum_{n=0}^\infty\frac{H_{n}(t,\ud)}{\la^n}q_{n+1}(1)=\la,\]
which ensures that $g_\la$ is non-zero. Note that
$g_\la^{(N)}:=\displaystyle\sum_{n=0}^N \frac{H_n(t,\ud)}{\la^n}\chi_{[\um,\sigma^n(\ud)]}\in\cF$
for $N\geq1$. Since
\[\begin{split}
&\Biggl|\Biggl|\sum_{n=N+1}^\infty\frac{H_n(t,\ud)}{\la^n}\chi_{[\um,\sigma^n(\ud)]}\Biggr|\Biggr|_{\cB} \\
&\leq\Biggl|\sum_{n=N+1}^\infty\frac{H_n(t,\ud)}{\la^n}\chi_{[\um,\sigma^n(\ud)]}\Biggr|_{\infty}
+\sum_{n=N+1}^\infty\frac{H_n(t,\ud)}{|\la|^n}\var(\chi_{[\um,\sigma^n(\ud)]}) \\
&\leq 3\cdot\sum_{n=N+1}^\infty\frac{H_n(t,\ud)}{|\la|^n}\to0
\end{split}
\]
as $N\to\infty$, we have $g_\la^{(N)}\to g_\la$ as $N\to\infty$ in the sense of the norm $||\cdot||_{\cB}$. Since the space $\cF$ is closed, we obtain $g_\la\in\cF$. 

By Lemma \ref{key lemma} and continuity of the operator $\cL_t$ on $\cF$, we have
\[\begin{split}
\cL_t g_\la&=\cL_t\Biggl(\sum_{n=0}^\infty\frac{H_n(t,\ud)}{\la^n}\chi_{[\um,\sigma^n(\ud)]}\Biggr) \\
&=\sum_{n=0}^\infty\frac{H_n(t,\ud)}{\la^n}\Bigl(q_{n+1}(1)\chi_{[\um,\ud]}+l(q_{n+1}(1))\chi_{[\um,\sigma^{n+1}(\ud)]}\Bigr) \\
&=\sum_{n=0}^\infty\frac{H_n(t,\ud)q_{n+1}(1)}{\la^n}\chi_{[\um,\ud]}
+\sum_{n=0}^\infty\frac{H_{n+1}(t,\ud)}{\la^n} \chi_{[\um,\sigma^{n+1}(\ud)]}\\
&=\la\sum_{n=1}^\infty\frac{H_{n-1}(t,\ud)q_{n}(1)}{\la^n}\chi_{[0,\ud]}
+\la\sum_{n=1}^\infty\frac{H_n(t,\ud)}{\la^n} \chi_{[\um,\sigma^{n}(\ud)]} \\
&=\la\Biggl(\phi_t(\la^{-1})\cdot\chi_{[\um,\ud]}+\sum_{n=1}^\infty\frac{H_n(t,\ud)}{\la^n}\chi_{[\um,\sigma^{n}(\ud)]}\Biggr) \\
&=\la g_\la,
\end{split}
\]
which gives the conclusion.
\end{proof}

By Proposition \ref{phi_t} (3) and Theorem \ref{Main B}, we have the following:

\begin{theorem}\label{determinant}
Let $t\in\R$. For $\la\in\C$ with $|\la|>r_t$, $\la$ is an isolated eigenvalue of $\cL_t$ if and only if $\la^{-1}$ is a zero of $1-\phi_t(z)$. 
\end{theorem}

As an application of the above theorem, we obtain the Ruelle-Perron-Frobenius theorem for $\cL_t$ for any $t\in\R$ and exponential decay of correlation functions with some explicit formulae, including the optimal upper bound of the decay rate of correlation functions. 

\begin{theorem}\label{Main C}
    Let $t\in\R$ and let $\cL_t$ be the transfer operator of the potential $t\chi_{C_1}$, where $C_1=\{(x_i)_{i=1}^\infty\in X_\be;\ x_1=1\}$. Then we have the following.

(1) Let $\la\in\C$ be an isolated eigenvalue of $\cL_t$ with $|\la|>r_t$. Then the geometric multiplicity of $\la$ is $1$ and for the corresponding eigenfunctional $\nu$ with $\nu(1)=1$, we have 
\begin{equation}\label{3-8(1)}
    \nu(\chi_{[\um,\ua]})=\sum_{n=1}^\infty\frac{a_n H_{n-1}(t,\ua)}{\la^n}
\end{equation}
for $\ua=(a_i)_{i=1}^\infty\in X_\be$.
Furthermore, any corresponding eigenfunction $h$ is of the form
\[h=C\cdot\sum_{n=0}^\infty\frac{H_n(t,\ud)}{\la^n}\chi_{[\um,\sigma^n(\ud)]},\]
where $C\in\C$.

(2) The positive eigenvalue $\la_t>r_t$ of $\cL_t$ is simple. The corresponding eigenfunctional $\nu_t$ with $\nu_t(1)=1$ is expressed as a non-atomic probability measure on $X_\be$ satisfying (\ref{3-8(1)}) for $\la=\la_t$. In addition, $\mu_t:=h_t\nu_t$ is $\sigma$-invariant probability measure on $X_\be$, where 
\[h_t=\frac{1}{F(t)}\sum_{n=0}^\infty\frac{H_n(t,\ud)}{\la_t^n}\chi_{[\um,\sigma^n(\ud)]}.\]
Here, $F(t)$ is the normalizing constant:
\begin{equation}\label{normalizing}
    F(t)=\sum_{n=1}^\infty\frac{nq_n(1) H_{n-1}(t,\ud)}{\la_t^n}.
\end{equation}

(3) The positive eigenvalue $\la_t$ is a unique leading eigenvalue of $\cL_t$, that is, we have $|\la|<\la_t$ for any other isolated eigenvalue $\la\in\C$ with $|\la|>r_t$. In addition, there is a positive constant $K$ such that 
\[\Biggl|\int_{X_\be}f\cdot g\circ\sigma^n d\mu_t-\int_{X_\be}fd\mu_t\cdot\int_{X_\be}g d\mu_t\Biggl|\leq K||f||_\cB ||g||_{L^1(\mu_t)}\alpha^n\]
for $n\geq1$, $f\in\cF$ and $g\in L^1(\mu_t)$, where $\mu_t=h_t\nu_t$ is a $\sigma$-invariant probability measure on $X_\be$ and 
\[\al=\max\Bigl\{r_t,
    \max\{|\la|; 1-\phi_t(\la^{-1})=0 \text{ with } \la\neq\la_t\} \Bigr\}\in[r_t,\la_t).
\]
In particular, the measurable dynamics $(\sigma,\mu_t)$ is exponentially mixing. 
\end{theorem}

\begin{proof}
    (1) Since the dimension of the eigenspace with respect to an isolated eigenvalue of $\cL_t^*$ is equal to that of $\cL_t$ (see e.g., \cite{Ba,Ha}), by Theorem \ref{Main A}, we have that the geometric multiplicity of $\la$ is $1$. The desired explicit formula for the eigenfunctional $\nu$ with $\nu(1)=1$ and that for an eigenfunction $h$ follow from Theorem \ref{Main A} and Theorem \ref{Main B}, respectively.

    (2) We show that the geometric multiplicity of $\la_t$ is equal to its algebraic multiplicity. Since $h_t$ is a positive function on $X_\be$, we know that $\nu_t(h_t)>0$. If there is a non-zero function $f\in\cF$ such that $h_t=(\la_\phi I-\cL_\phi)f$, where $I$ denotes the identity map on $\cF$, we obtain that $\nu_t(h_t)=\la_t\nu_t(f)-\nu_t(\cL_t f)=0$, which gives the contradiction. The fact that $\nu_t$ is represented by a non-atomic probability measure on $X_\be$ is a consequence of Proposition \ref{3-4} (3). The formula for $h_t$ is a consequence of Theorem \ref{Main B}. The fact that $\mu_t=h_t\nu_t$ is $\sigma$-invariant is an immediate consequence of the equations $\cL_th_t=\la_th_t$ and $\cL_t^*\nu_t=\la_t\nu_t$. The formula for $F(t)$ follows from the direct calculation:
    \[\begin{split}
    \nu_t\Biggl(\sum_{n=0}^\infty\frac{H_n(t,\ud)}{\la_t^n}\chi_{[\um,\sigma^n(\ud)]}\Biggr)
    &=\sum_{n=0}^\infty\frac{H_n(t,\ud)}{\la_t^n}\nu_t(\chi_{[\um,\sigma^n(\ud)]}) \\
    &=\sum_{n=0}^\infty\frac{H_n(t,\ud)}{\la^n}\sum_{m=1}^\infty \frac{q_{n+m}(1)H_{m-1}(t,\sigma^n(\ud))}{\la_t^m} \\
    &=\sum_{n=0}^\infty\sum_{m=1}^\infty \frac{q_{n+m}(1)H_{n+m-1}(t,\ud)}{\la_t^{n+m}}\\
    &=\sum_{n=1}^\infty\frac{nq_n(1) H_{n-1}(t,\ud)}{\la_t^n}.
    \end{split}\]

    (3) By Proposition \ref{phi_t} (4) and Theorem \ref{determinant}, for an isolated eigenvalue $\la\in\C$ with $\la\neq\la_t$, we have $\la_t^{-1}<|\la^{-1}|$, which yields $|\la|<\la_t$. By the spectral decomposition (\ref{decomposition}) and the fact that $\la_t$ is a unique leading eigenvalue of $\cL_t$, there is a positive constant $K'$ such that 
    \[\Bigg|\frac{\cL_t^n(fh_t)}{\la_t^n}-\Bigl(\int_{X_\be}fh_t d\nu_t\Bigr)\cdot h_t\Biggr|_{\infty}\leq K'||f||_{\cB}\alpha^n\]
    for any $f\in\cF$ and all $n\geq1$ (see e.g., \cite[Theorem 1.6]{Ba}). This shows that 
    \begin{align*}
        &\Biggl|\int_{X_\be}f g\circ\sigma^n d\mu_t-\int_{X_\be}fd\mu_t\cdot\int_{X_\be}g d\mu_t\Biggl| \\
        &\leq\int_{X_\be}\Bigg|\frac{\cL_t^n(fh_t)}{\la_t^n}-\Bigl(\int_{X_\be}fh_t d\nu_t\Bigr)\cdot h_t\Biggr|_\infty g\ d\nu_t
        \leq K||f||_\cB ||g||_{L^1(\mu_t)}\alpha^n
    \end{align*}
    for any $f\in\cF$ and all $n\geq1$, where $K=K'/\inf{h_t}>0$. This gives the conclusion.
    \end{proof}

\end{section}

\begin{section}{Hausdorff dimension of frequency sets}

In this section, we give an exact formula for the Hausdorff dimension of frequency sets for $1<\be<2$ as an application of the results in the previous section. 

We start with the definitions of $g$-measures and equilibrium states. For a positive function $g\in\cF$ with $\cL_{\log g}1=1$ a Borel probability measure $m$ on $X_\be$ is called a $g$-measure if $\cL_{\log g}^*m=m$, i.e., $\displaystyle{\int_{X_\be}\cL_{\log g}f dm=\int_{X_\be} f dm}$ for any $f\in\cF$. For a function $\phi\in\cF$, a Borel probability measure $\mu$ on $X_\be$ is called an equilibrium state with respect to $\phi$ if it is $\sigma$-invariant and satisfies
\[\hat{P}(\phi):=\sup_{m\in M(X_\be)}\Biggl\{h_m(\sigma)+\int_{X_\be}\phi\ dm\Biggr\}=h_\mu(\sigma)+\int_{X_\be}\phi d\mu,\]
where $M(X_\be)$ denotes the set of all $\sigma$-invariant probability measures on $X_\be$ and $h_m(\sigma)$ denotes the measure-theoretic entropy for $(\sigma,m)$, where $m\in M(X_\be)$. We note that if a function $\phi$ is continuous on $X_\be$ then $\hat{P}(\phi)$ is equal to the topological entropy $P(\phi)$:
\[P(\phi)=\lim_{n\to\infty}\frac{1}{n}\log\sum_{A\in \cC_n}\exp\Biggl(\sup_{\ux\in A}\sum_{i=0}^{n-1}\phi(\sigma^i\ux)\Biggr),\]
which is the consequence of the variational principle (see e.g., \cite[Theorem 9.10]{Wa}). 

The following result for one-sided shift spaces is due to Ledrappier \cite{Le}, whose generalization to topological dynamical systems is given by Walters \cite{Wa1}. We omit the proof since it is the same as that given in \cite{Le}. 

\begin{lemma}\label{4-1}
Let $g\in\cF$ be a positive function with $\cL_{\log g}1=1$. For a Borel probability measure $m$ on $X_\be$ the following are equivalent:

    (1) $\cL_{\log g}^*m=m$.

    (2) $m$ is $\sigma$-invariant and it is an equilibrium state with respect to $\log g$.  
Furthermore, if $m$ is a $g$-measure then $h_m(\sigma)+\int_{X_\be} \log g\ dm=0$.
\end{lemma}

We need the following property: 

\begin{proposition}\label{4-2}
    For $t\in\R$ let $(\la_t,h_t,\nu_t)$ be the triple as given in Theorem \ref{Main C}(2). Then the $\sigma$-invariant measure $\mu_t=h_t \nu_t$ is a unique equilibrium state with respect to the potential $t\chi_{C_1}$. Furthermore, $P(t\chi_{C_1})=\log \la_t$. 
    \end{proposition}
\begin{proof}
We follow a standard argument using the spectral gap property of $\cL_t$ for $t\in\R$. For $t\in\R$ set  $g_t=e^{t\chi_{C_1}}h_t/\la_t(h_t\circ\sigma)$. 
Since $g_t>0$, $\inf h_t>0$ and $h_t/(h_t\circ\sigma)\in\cF$, we have $g_t\in\cF$, which yields that $\log g_t\in\cF$. 
In addition, 
\[\cL_{\log g_t}1(\ux)=\sum_{\ux=\sigma\uy}\frac{e^{t\chi_{C_1}(\uy)}h_t(\uy)}{\la_t(h_t(\sigma \uy))}=\frac{\cL_th_t(\ux)}{\la_th_t(\ux)}=1\]
for $\ux\in X_\be$ and 
\begin{align*}
\int_{X_\be}\cL_{\log g_t}f d\mu_t
&=\int_{X_\be}\frac{\cL_t(fh_t)}{\la_th_t}h_t d\nu_t \\
&=\int_{X_\be}\frac{fh_t}{\la_t} d(\cL_t^*\nu_t) \\
&=\int_{X_\be}f d\mu_t
\end{align*}
for $f\in\cF$. This shows that $\mu_t$ is a $g_t$-measure. By Lemma \ref{4-1}, we have that $\mu_t$ is an equilibrium state for $t\chi_{C_1}+\log h_t -\log \la_t -\log h_t\circ\sigma$. In fact 
\[\int_{X_\be}(t\chi_{C_1}+\log h_t -\log \la_t -\log h_t\circ\sigma)d\mu_t=\int_{X_\be}t\chi_{C_1} d\mu_t-\log \la_t,\]
which yields that the measure $\mu_t$ is an equilibrium state for $t\chi_{C_1}-\log \la_t$. Since  $\log \la_t$ is a constant, we have that $\mu_t$ is an equilibrium state for $t\chi_{C_1}$ (see e.g., Theorem 9.4 (iv) in \cite{Wa}). Furthermore, by the above equality and Lemma \ref{4-1} (2) we have
\[0=\hat{P}(\log g_t)=P(t\chi_{C_1})-\log \la_t,\]
which yields $P(t\chi_{C_1})=\log\la_t$. 

Let $m$ be an equilibrium state for $t\chi_{C_1}$. Then we can see that it is also an equilibrium state for $\log g_t=t\chi_{C_1}+\log h_t -\log \la_t -\log h_t\circ\sigma$, which yields that $m$ is a $g_t$-measure. 
By Theorem \ref{Main C} (3), the transfer operator $\cL_t$ has a spectral gap property, i.e., there is $\theta<\la_t$ such that 
\[\cL_t^nf=\la_t^n\int_{X_\be}f\ d\nu_\be \cdot h_t+O(\theta^n)\]
for any $f\in\cF$. Hence 
\begin{align*}
\int_{X_\be} f dm
&=\int_{X_\be} \cL_{\log g_t}^nf dm \\
&= \int_{X_\be} \frac{\cL_t^n(fh_t)}{\la_t^n h_t} dm\\
&=\int_{X_\be}\frac{\int_{_{X_\be}}fh_t\ d\nu_t\cdot h_t+O(\theta^n/\la_t^n)}{h_t} \\
&\to \int_{X_\be} fh_t d\nu_t=\int_{X_\be} f d\mu_t
\end{align*}
as $n\to\infty$ for any $f\in\cF$. Since $\cF$ includes the set of all indicator functions of any intervals in $X_\be$, we obtain $m=\mu_t$. 
\end{proof}

In the rest of this section, we give a formula for the dimension spectrum of the frequency set 
\[\Lambda_\alpha:=\Bigl\{(x_i)_{i=1}^\infty\in X_\be\ ; \lim_{n\to\infty}\frac{1}{n}\sum_{i=1}^n x_i=\alpha\Bigr\}\]
for $\al=\al(t)=\int_{X_\be}\chi_{C_1}d\mu_t$ for $t\in\R$, where $\mu_t$ is a unique equilibrium state for $t\chi_{C_1}$. The following conditional variational principle is given in \cite[Proposition 6]{Li} as a consequence of some results in \cite{Pf-Su} (see also \cite{Th}), which is a key tool for the proof of the main result (for more general cases, see \cite{Cl,Ho}).

\begin{theorem}\label{vp}
    For $\alpha\in[0,1]$ we have
    \[\dim_{H}\Lambda_\al=\frac{\sup_{\mu\in M(X_\be)}\{h_\mu(\sigma);\ \int_{X_\be}\chi_{C_1} d\mu=\al\}}{\log \be}.\] 
     Here, we regard $\sup\emptyset$ as $0$. 
\end{theorem}

The next lemma relates the above formula to the equilibrium state $\mu_t$. 

\begin{lemma}\label{4-4}
    For $t\in\R$ let $\mu_t$ be an equilibrium state with respect to  $t\chi_{C_1}$. Then \[\dim_{H}\Lambda_{\al(t)}=\frac{h_{\mu_t}(\sigma)}{\log\be},\]
    where $\al(t)=\int_{X_\be}\chi_{C_1}d\mu_t$.
    \end{lemma}

\begin{proof}
    Let $m$ be any $\sigma$-invariant probability measure on $X_\be$ with $m\neq \mu_t$ and $m(\chi_{C_1})=\al(t)$. Since $\mu_t$ is a unique equilibrium state for $t\chi_{C_1}$, we know that 
    \[h_{\mu_t}(\sigma)+\int_{X_\be}t\chi_{C_1} d\mu_t> h_m(\sigma)+ \int_{X_\be}t\chi_{C_1}dm.\]
Since $\displaystyle{\al(t)=\int_{X_\be}\chi_{C_1} d\mu_t}=\int_{X_\be}\chi_{C_1} dm$, we obtain $h_{\mu_t}(\sigma)>h_m(\sigma)$. By Theorem \ref{vp}, we have the conclusion.
\end{proof}

\begin{theorem}\label{Main D}
For $t\in\R$ let $\phi_t(z)$ be the power series as defined in (\ref{power series}). Let $\la_t$ be the inverse of a unique positive zero of $1-\phi_t(z)$ and let $\mu_t$ be an equilibrium state with respect to  $t\chi_{C_1}$. Then 
\begin{align}\label{alphat}
\al(t)&=\notag
\int_{X_\be}\chi_{C_1}d\mu_t \\ 
&=\frac{1}{F(t)}\sum_{n=1}^\infty\Bigl(\sum_{i=1}^n q_i(1)-1\Bigr)\frac{q_{n}(1)H_{n-1}(t,\ud)}{\la_t^n},
\end{align}
where $F(t)$ is the normalizing constant as defined in (\ref{normalizing}) and 
    \[\dim_H(\Lambda_{\alpha(t)})=\frac{\log \la_t-t\al(t)}{\log\be}.\]
In particular, if $\be\in(1,2)$ is simple, i.e., there is a positive integer $n_0\geq2$ such that $\tau_\be^{n_0-1}(1)=1/\be$, the function $\al(t)$ can be given by
\[\al(t)=\frac{1}{\widehat F(t)}\Biggl(\sum_{l=1}^{n_0-1}\frac{g_{l+1}(1)H_l(t,(g_m(1))_{m=1}^\infty)}{\la_t^l}\sum_{n=1}^{n_0-l} \frac{g_{n+l}(1)H_{n-1}(t, (g_{m+l}(1))_{m=1}^\infty)}{\la_t^n}\Biggr),\]
where 
\[\widehat F(t)=\sum_{l=0}^{n_0-1}\frac{H_l(t,(g_m(1))_{m=1}^\infty)}{\la_t^l}\sum_{n=1}^{n_0-l} \frac{g_{n+l}(1)H_{n-1}(t, (g_{m+l}(1))_{m=1}^\infty)}{\la_t^n}.\]

\end{theorem}

\begin{proof}
    Let $t\in\R$. By Lemma \ref{4-4}, we know that 
    \[\dim_H(\Lambda_{\al(t)})=\frac{h_{\mu_t}(\sigma)}{\log\be}.\]
    By Proposition \ref{4-2} and the variational principle, we have that 
    \[\log \la_t=P(t\chi_{C_1})=h_{\mu_t}(\sigma)+t\al(t),\]
    which gives the equality
    \[\dim_H(\Lambda_{\alpha(t)})=\frac{\log \la_t-t\al(t)}{\log\be}.\]
    
The formula for $\al(t)$ is given by the direct calculation: 
\begin{align*}
    \int_{X_\be}\chi_{C_1} d\mu_t
    &=\frac{1}{F(t)}\sum_{n=0}^\infty \frac{H_{n}(t,\ud)}{\la_t^n}\nu_t(\chi_{C_1}\cdot \chi_{[\um,\sigma^n(\ud)]}) \\
    &=\frac{1}{F(t)}\sum_{n=0}^\infty \frac{H_{n}(t,\ud)}{\la_t^n}q_{n+1}(1)\nu_t(\chi_{[1\um,\sigma^n(\ud)]}) \\
    &=\frac{1}{F(t)}\sum_{n=0}^\infty \frac{H_{n}(t,\ud)}{\la_t^n}q_{n+1}(1)(\nu_t(\chi_{[\um,\sigma^n(\ud)]})-\nu_t(\chi_{[\um,1\um]})) \\
    &=\frac{1}{F(t)}\sum_{n=0}^\infty \frac{H_{n}(t,\ud)}{\la_t^n}q_{n+1}(1)\sum_{m=1}^\infty\frac{q_{n+m}(1)H_{m-1}(t,\sigma^n(\ud))}{\la_t^m} \\
    &\ \ -\frac{1}{F(t)}\sum_{n=0}^\infty \frac{H_{n}(t,\ud)}{\la_t^{n+1}}q_{n+1}(1) \\
    &=\frac{1}{F(t)}\sum_{n=0}^\infty\sum_{m=1}^{\infty}\frac{q_{n+m}(1) H_{n+m-1}(t,\ud)}{\la_t^{n+m}}q_{n+1}(1)-\frac{1}{F(t)} \\
    &=\frac{1}{F(t)}\sum_{n=1}^\infty\Bigl(\sum_{i=1}^n q_i(1)\Bigr)\frac{q_{n}(1)H_{n-1}(t,\ud)}{\la_t^n}-\frac{1}{F(t)} \\
    &=\frac{1}{F(t)}\sum_{n=1}^\infty\Bigl(\sum_{i=1}^n q_i(1)-1\Bigr)\frac{q_{n}(1)H_{n-1}(t,\ud)}{\la_t^n}.
\end{align*}

Assume that there is a positive integer $n_0\geq2$ such that $\tau_\be^{n_0-1}(1)=1/\be$. We know that $\ud=\overline{q_1(1)\dots q_{n_0}(1)}^\infty$, $q_i(1)=g_i(1)$ for $1\leq i< n_0$, $0=q_{n_0}(1)<g_{n_0}(1)=1$ and $g_i(1)=0$ for $i\geq n_0+1$. Set 
\[K(t)=\frac{1}{1-\exp\Bigl(t\cdot\sum_{i=1}^{n_0}q_i(1)\Bigr)/\la_t^{n_0}}.\] By Proposition \ref{3-4} (1), we have 
\begin{align*}
    F(t)&=\sum_{n=1}^\infty \frac{n q_n(1) H_{n-1}(t, \ud)}{\la_t^n} \\
    &=K(t)\cdot\sum_{l=0}^{n_0-1}\frac{H_l(t,(q_m(1))_{m=1}^\infty)}{\la_t^l}\sum_{n=1}^\infty \frac{q_n(\tau_\be^l(1))H_{n-1}(t, (q_m(\tau_\be^l(1)))_{m=1}^\infty)}{\la_t^n} \\
    &=K(t)\cdot\sum_{l=0}^{n_0-1}\frac{H_l(t,(g_m(1))_{m=1}^\infty)}{\la_t^l}\sum_{n=1}^\infty \frac{g_n(\tau_\be^l(1))H_{n-1}(t, (g_m(\tau_\be^l(1)))_{m=1}^\infty)}{\la_t^n} \\
    &=K(t)\cdot\sum_{l=0}^{n_0-1}\frac{H_l(t,(g_m(1))_{m=1}^\infty)}{\la_t^l}\sum_{n=1}^{n_0-l} \frac{g_{n+l}(1)H_{n-1}(t, (g_{m+l}(1))_{m=1}^\infty)}{\la_t^n} \\&=K(t)\widehat{F}(t).
    \end{align*}
Similarly, we obtain 
\begin{align*}
    Q(t)&:=\sum_{n=1}^\infty \Bigl(\sum_{i=1}^nq_i(1)\Bigr)\frac{ q_n(1) H_{n-1}(t, \ud)}{\la_t^n} \\
    &=K(t)\cdot\sum_{l=0}^{n_0-1}\frac{q_{l+1}(1)H_l(t,(q_m(1))_{m=1}^\infty)}{\la_t^l}\sum_{n=1}^\infty \frac{q_n(\tau_\be^l(1))H_{n-1}(t, (q_m(\tau_\be^l(1)))_{m=1}^\infty)}{\la_t^n} \\
    &=K(t)\cdot\sum_{l=0}^{n_0-2}\frac{g_{l+1}(1)H_l(t,(g_m(1))_{m=1}^\infty)}{\la_t^l}\sum_{n=1}^\infty \frac{g_n(\tau_\be^l(1))H_{n-1}(t, (g_m(\tau_\be^l(1)))_{m=1}^\infty)}{\la_t^n} \\
    &=K(t)\cdot\sum_{l=0}^{n_0-2}\frac{g_{l+1}(1)H_l(t,(g_m(1))_{m=1}^\infty)}{\la_t^l}\sum_{n=1}^{n_0-l} \frac{g_{n+l}(1)H_{n-1}(t, (g_{m+l}(1))_{m=1}^\infty)}{\la_t^n}.
\end{align*}
Hence 
\begin{align*}
    &\al(t)
    =\frac{Q(t)-1}{F(t)} \\
    &=\frac{1}{\widehat F(t)}\Biggl(\sum_{l=0}^{n_0-2}\frac{g_{l+1}(1)H_l(t,(g_m(1))_{m=1}^\infty)}{\la_t^l}\sum_{n=1}^{n_0-l} \frac{g_{n+l}(1)H_{n-1}(t, (g_{m+l}(1))_{m=1}^\infty)}{\la_t^n}-\frac{1}{K(t)}\Biggr)\\
    &=\frac{1}{\widehat F(t)}\Biggl(\sum_{l=1}^{n_0-2}\frac{g_{l+1}(1)H_l(t,(g_m(1))_{m=1}^\infty)}{\la_t^l}\sum_{n=1}^{n_0-l} \frac{g_{n+l}(1)H_{n-1}(t, (g_{m+l}(1))_{m=1}^\infty)}{\la_t^n}+1 \\&\ \ -1+\exp\Bigl(\sum_{i=1}^{n_0}q_i(1)/\la_t^{n_0}\Bigr)\Biggr) \\
    &=\frac{1}{\widehat F(t)}\Biggl(\sum_{l=1}^{n_0-1}\frac{g_{l+1}(1)H_l(t,(g_m(1))_{m=1}^\infty)}{\la_t^l}\sum_{n=1}^{n_0-l} \frac{g_{n+l}(1)H_{n-1}(t, (g_{m+l}(1))_{m=1}^\infty)}{\la_t^n}\Biggr),
\end{align*}
which gives the conclusion.

\end{proof}
 
\end{section}

\begin{section}{Analytic properties of the pressure function $P(t \chi_{C_1})$}
This section is devoted to investigating some properties of the leading eigenvalue $\la_t$ of $\cL_t$ for $t\in\R$, including its analyticity. As we see in the proof of the next proposition, there is the minimal value $c_\be$ such that if $t>c_\be$ then $\La_{\al(t)}=\emptyset$, where  $c_\be=\sup_{\mu\in M(X_\be)}\int_{X_\be}\chi_{C_1} d\mu$. In addition, we show the analyticity of the function $p\in(0,c_\be)\mapsto t(p)\in\R$, where $t(p)\in\R$ is the real number satisfying $\alpha(t(p))=p$. 

\begin{proposition}\label{5-1}
    For $t\in\R$ let $\la_t$ be the leading eigenvalue of $\cL_t$. 
    Then we have the following: 

    (1) The function $t\in\R\mapsto\la_t\in(1,\infty)$ is a strictly increasing function satisfying
    \[\lim_{t\to-\infty}\la_t=1 \ \ \ \ \text{ and }\ \ \ \lim_{t\to\infty}\la_t=\infty.\]

    (2) The function $t\in\R\mapsto\la_t\in(1,\infty)$ can be extended to an analytic function on some region $\C$ including $\R$. In particular, the function $t\in\R\mapsto\la_t\in(1,\infty)$ is real-analytic on $\R$. 

    (3) The pressure $P(t):=P(t\chi_{C_1})$ is equal to $\log \la_t$. Furthermore, $\displaystyle{\alpha(t):=\frac{dP(t)}{dt}}=\int_{X_\be}\chi_{C_1} d\mu_t$, where $\mu_t$ is a unique equilibrium state of $t \chi_{C_1}$. 

    (4) The function $\alpha(t)$ is positive and strictly increasing on $t\in\R$ satisfying
    \[\lim_{t\to-\infty}\al(t)=0,\ \ \lim_{t\to\infty}\al(t)=c_\be:=\sup_{\mu\in M(X_\be)}\int_{X_\be}\chi_{C_1} d\mu. \]
    In addition, we have $\La_\al=\emptyset$ if $\al\notin[0,c_\be]$.

    (5) For $p\in(0, c_\be)$ let $t(p)\in\R$ be the real number satisfying $\alpha(t(p))=p$. Then the function $p\in(0,c_\be)\mapsto t(p)\in\R$ is real-analytic. 

    (6) The function $\al\in(0,c_\be)\mapsto \dim_H\Lambda_\alpha\in[0,1]$ is real analytic and concave.
    \end{proposition}
\begin{proof}
    (1) By Proposition \ref{phi_t} (3), we know that $\la_t$ is the inverse of a unique positive zero of $1-\phi_t(z)$, where $\phi_t(z)$ is the power series as defined in (\ref{power series}). Each coefficient of $\phi_t(z)$ is $q_n(1) H_{n-1}(t,\ud)\geq0$ for $n\geq1$. By definition $H_{n-1}(t,\ud)<H_{n-1}(t',\ud)$ for $t,t'\in\R$ with $t<t'$ if $n>1$. Since there are infinitely many $n$ such that $q_n(1)=1$, we have $\phi_t(r)<\phi_{t'}(r)$ for $r\in[0,r_t^{-1})$. This shows that $1/\la_{t'}<1/\la_{t}$ for $t,t'\in\R$ with $t<t'$, which yields $\la_t<\la_{t'}$.

    We show that $\lim_{t\to -\infty}\la_t=1$. Since 
    $q_n (1)H_{n-1}(t,\ud)>0$ for $n\geq1$ such that $q_n(1)=1$, we obtain 
    \[1=\phi_t(\la_t^{-1})=\frac{1}{\la_t}+\sum_{n=2}^\infty \frac{q_n(1) H_{n-1}(t,\ud)}{\la_t^n}>\frac{1}{\la_t},\] which gives $\la_t>1$ for $t\in\R$. Let $t<0$. Since $e^t<1$ and $\la_t>1$ we have
    \[1=\phi_t(\la_t^{-1})\leq\frac{1}{\la_t}+\frac{e^t}{\la_t^2}+\cdots+\frac{e^{(n-1)t}}{\la_t^n}+\cdots\leq \frac{1}{\la_t-e^t}.\]
    This gives $\la_t-e^t\leq1$, which yields $\la_t\leq 1+e^t$. Hence we obtain 
    \[1\leq \lim_{t\to-\infty}\la_t\leq\lim_{t\to-\infty}(1+e^t)=1,\]
    as desired. 

    The fact that $\lim_{t\to \infty}\la_t=\infty$ follows from the inequality
    \[1=\phi_t(\la_t^{-1})\geq \frac{1}{\la_t}+\frac{e^t}{\la_t^k},\]
    where $k$ is the minimal integer greater than $1$ with $q_k(1)=1$, which yields 
    \[\la_t^k\geq\la_t^{k-1}+e^t>e^t\to\infty\]
     as $t\to\infty$.

    (2) We define the formal power series $\phi(w,z)$ by
    \[\phi(w,z)=\sum_{n=1}^\infty q_n(1) H_{n-1}(w,\ud)z^n,\]
    where
    \[H_{n}(w,\ud)=\begin{cases}
        \exp\Bigl(\sum_{i=1}^{n}q_i(1)\cdot w\Bigr), & (n\geq1) \\
        1.&(n=0)
    \end{cases}\]
We note that
\[\Biggl|\exp\Bigl(\sum_{i=1}^{n}q_i(1)\cdot w\Bigr)\Biggr|=\exp\Bigl(\sum_{i=1}^{n}q_i(1)\cdot Re(w)\Bigr)\]
for $w\in\C$. Hence the power series $\phi(w,z)$ converges absolutely in 
\[R=\Biggl\{(\omega,z)\in\C^2; |z|<\exp\Bigl(-\limsup_{n\to\infty}\frac{1}{n}\sum_{i=1}^nq_i(1) \cdot Re(w)\Bigr)\Biggr\}.\]
Since $\phi(w,z)$ is continuous on $R$ and analytic in each variable $w,z$ separately, we have that $\phi(w,z)$ is analytic on $R$ by Osgood's lemma (see e.g., \cite{Gu-Ro}). 

Set $\Phi(w,z)=1-\phi(w,z)$ for $(w,z)\in R$. Then $\Phi(t,\la_t^{-1})=0$ for any $t\in\R$ and 
\[\frac{\partial\Phi(w,z)}{\partial z}\Biggl|_{(w,z)=(t,\la_t^{-1})}=-\sum_{n=1}^\infty \frac{nq_n(1)H_{n-1}(t,\ud)}{\la_t^{n-1}}<0.\]
By the holomorphic implicit function theorem, there are an open region $U_t\times V_t\ni(t,\la_t^{-1})$ and a holomorphic function $\gamma_t:U_t\to\C$ with $1-\phi(z,\gamma_t(z))=0$ and $|\gamma_t(z)|\neq0$ for $z\in U_t$. Define $\Gamma:\cup_{t\in\R}U_t\to\C$ by $\Gamma(z)=1/\gamma_t(z)$ for $z\in U_t$. We note that this holomorphic function is well-defined when $z\in U_t\cap U_{t'}$ for $t\neq t'$ by the identity theorem. By the definition of the function $\Gamma$, we have $\Gamma(t)=\la_t$ for $t\in\R$, which gives the desired function. 

(3) The fact that the pressure is given by $P(t)=P(t\chi_{C_1})=\log\la_t$ for $t\in\R$ follows from Proposition \ref{4-2}. Let $t_0\in\R$. Then the variational principle yields 
\begin{align*}
    P(t)&=\sup_{\mu\in M(X_\be)}\Bigl\{h(\mu)+t_0\int_{X_\be}\chi_{C_1}d\mu +(t-t_0)\int_{X_\be}\chi_{C_1}d\mu \Bigr\} \\
    &\geq h(\mu_{t_0})+t_0\int_{X_\be}\chi_{C_1}d\mu_{t_0} +(t-t_0)\int_{X_\be}\chi_{C_1}d\mu_{t_0} \\
    &=P(t_0)+(t-t_0)\int_{X_\be}\chi_{C_1}d\mu_{t_0}
\end{align*}
for $t\in\R$. Then we have
\[P(t)-P(t_0)\geq (t-t_0)\int_{X_\be}\chi_{C_1}d\mu_{t_0}.\]
This shows that 
\[\frac{P(t)-P(t_0)}{t-t_0}\geq\int_{X_\be}\chi_{C_1}d\mu_{t_0}\]
if $t>t_0$ and 
\[\frac{P(t)-P(t_0)}{t-t_0}\leq\int_{X_\be}\chi_{C_1}d\mu_{t_0}\]
if $t<t_0$. Since $P(t)$ is differentiable at $t_0$ by Item (2), we obtain the desired result. 

(4) Note that $\al(t)$ is the derivative of the pressure function $P(t)$ by Item (3). The fact that $\al(t)$ is strictly positive follows from that $P(t)=\log \la_t$ is strictly increasing by Item (1). Since $P(t)$ is real-analytic by Item (2) so is $\al(t)$. The pressure $P(t)$ is a convex function (see e.g., \cite[Theorem 9.7(v)]{Wa}), which yields $\al'(t)\geq0$. Assume that $\al(t)$ is not strictly increasing, i.e., there are $t, t'\in\R$ with $t<t'$ such that $\al(t)=\al(t')$. Then by Item (2) and the identity theorem, $\al(t)$ is a constant function, which yields that $P(t)$ is a linear function of the form $at+b$, where $a,b\in\R$. Hence $\la_t=e^{at+b}$ for $t\in\R$, which contradicts Item (1). 

We have $\lim_{t\to-\infty}\al(t)=0$, $\lim_{t\to\infty}\al(t)=c_\be$ and $\La_\al=\emptyset$ for $\al\notin[0,c_\be]$ by an application of \cite[Corollary 2.9]{Cl}, which states that 
\[\Bigl\{\int_{X_\be}\chi_{C_1} d\mu \ ; \mu\in M(X_\be) \Bigr\}=[\lim_{t\to-\infty}\al(t), \lim_{t\to\infty}\al(t)]\]
and $\La_\al=\emptyset$ for $\al\notin[\lim_{t\to-\infty}\al(t), \lim_{t\to\infty}\al(t)]$. 
Since $\int_{X_\be}\chi_{C_1}d\delta_{\um}=0$ for the Dirac measure $\delta_{\um}$ at $\um$, we obtain 
\[\lim_{t\to-\infty}\al(t)=\int_{X_\be}\chi_{C_1}d\delta_{\um}=0\text{ and }\lim_{t\to\infty}\al(t)=\sup_{\mu\in M(X_\be)}\int_{X_\be}\chi_{C_1}d\mu.\]

(5) Since the function $p\in(0,c_\be)\mapsto t(p)\in(-\infty,\infty)$ is the inverse function of $\al(t)$ and $\al(t)$ can be extended to some analytic function which is  strictly increasing on $(-\infty,\infty)$ by Item (2), we have the conclusion by the holomorphic inverse function theorem. 

(6) By Items (3) and (5), together with Theorem \ref{Main D}, the function $\alpha\in(0,c_\be)\mapsto \dim\Lambda_\al\in[0,1]$ is given by a linear sum of real-analytic functions, which ensures that it is real-analytic. 
By Corollary 2.9 in \cite{Cl}, it is given by the minus of the Legendre transform of the pressure function $P(t)$. Since $P(t)$ is convex, we have that it is concave.
\end{proof}

In the following, we give upper and lower bounds for $c_\be$, whose value is given explicitly in the case where $\be$ satisfies $\be^{M+1}-\be^{M}-\dots-1=0$ or $\be^{M+1}-\be^M-1=0$ for some positive integer $M\geq1$.

\begin{theorem}\label{5-2} Let $1<\be<2$. 
    For the constant $\displaystyle c_\be=\sup_{\mu\in M(X_\be)}\int_{X_\be}\chi_{C_1} d\mu$, we have the following:

    (1) The constant $c_\be$ is non-decreasing for $\be\in(1,2)$ as a function of $\be$.

    (2) Let $\ud$ be the coefficient sequence of the quasi-greedy expansion of $1$. We denote $\ud$ by
    \[\ud=\overline{1}^{N}\overline{0}^{M}1\dots,\] where $N$, $M$ are positive integers and $\overline{a}^{k}$ denotes the word $\underbrace{a\dots a}_{k-times}$ for integers $k\geq1$ and  $a\geq0$. Then we have
    \[\frac{N-1}{N}\leq c_\be\leq\frac{N}{N+1}\]
    if $N\geq2$ and
    \[\frac{1}{M+2}\leq c_\be\leq\frac{1}{M+1}\]
    if $N=1$.

    (3) If $\be$ is a multinacci number, i.e., there is a positive integer $N$ such that 
    $\be^{N+1}-\be^N-\dots-1=0$ then $\ud=\overline{\overline{1}^{N}0}^\infty$ and $c_\be=N/(N+1)$. 

    (4) If $\be$ is a Parry number satisfying $\be^{M+1}-\be^M-1=0$ for some positive integer $M$ then $\ud=\overline{1\overline{0}^{M}}^\infty$ and $c_\be=1/(M+1)$. 
    \end{theorem}
\begin{proof}
    (1) We note that $X_\be\subset X_{\be^+}$ for $1<\be<\be^+\leq2$ (see e.g., \cite{It-Ta}). Then the set of $\sigma$-invariant probability measures on $X_\be$ is embedded to that on $X_{\be^+}$, which provides $c_\be\leq c_{\be^+}$.

    (2) The periodic point $\ux'=\overline{\overline{1}^{N-1}0}^\infty$ satisfies $\sigma^n\ux'\preceq\ud$ for all $n\geq0$. This shows that $\ux'\in X_\be$ by Proposition \ref{Parry1}. Hence the periodic measure $\displaystyle \mu'=\frac{1}{N}\sum_{i=0}^{N-1}\delta_{\sigma^i\ux'}$ is $\sigma$-invariant satisfying $\displaystyle\int_{X_\be}\chi_{C_1} d\mu'=\frac{N-1}{N}$, which yields that $c_\be\geq (N-1)/N$.

    We shall show that $c_\be\leq N/(N+1)$. Assume that there is a $\sigma$-invariant probability measure $\mu''$ such that $\displaystyle \int_{X_\be} \chi_{C_1} d\mu''>N/(N+1)$. By the Birkhoff ergodic theorem, we have
    \[\frac{1}{n}\sum_{i=1}^n x_i\to \chi^*_1(\ux)\]
    for $\mu''$-a.e. $\ux=(x_i)_{i=1}^\infty\in X_\be$ as $n\to\infty$ , where $\chi_1^*$ is a $\sigma$-invariant function whose integral by $\mu''$ on $X_\be$ is equal to that of $\chi_{C_1}$. Hence there are $\ux''=(x_i'')_{i=1}^\infty\in X_\be$ and a positive integer $k$ such that
    \[\frac{1}{k(N+1)}\sum_{i=1}^{k(N+1)} x_i''>\frac{N}{N+1},\]
     which yields 
    \[\sum_{i=1}^{k(N+1)} x_i''>kN.\]
    By the pigeon hole principle, there is a non-negative integer $k_0$ such that \[x_{k_0(N+1)+1}''\dots x_{(k_0+1)(N+1)}''=\overline{1}^{N+1}.\] This shows that $\sigma^{k_0(N+1)}\ux''\succ\ud$, 
    which contradicts $\sigma^n\ux''\preceq\ud$ for any $n\geq0$.

    Let $N=1$. Since the periodic point $\uy^{'}=\overline{1\overline{0}^{M+1}}^\infty$satisfies $\sigma^n\uy'\preceq\ud$ for all $n\geq0$, we have $\uy'\in X_\be$. The periodic measure $\displaystyle m'=\frac{1}{M+2}\sum_{i=0}^{M+1}\delta_{\sigma^i\uy'}$ satisfies $\displaystyle\int_{X_\be}\chi_{C_1} dm'=\frac{1}{M+2}$, which yields that $c_\be\geq 1/(M+2)$. 

    We shall show that $c_\be\leq 1/(M+1)$. Assume that there is a $\sigma$-invariant probability measure $m''$ such that  $\displaystyle \int_{X_\be} \chi_{C_1} dm''>1/(M+1)$. Then as in the same way of the case $N\geq2$, we have that there are $\uy''=(y''_i)_{i=1}^\infty\in X_\be$ and a positive integer $l\geq1$ such that 
    \[\frac{1}{l(M+1)}\sum_{i=1}^{l(M+1)} y_i''>\frac{1}{M+1},\] 
    which yields that there is a non-negative integer $l_0$ such that 
    \[\sum_{i=l_0(M+1)+1}^{(l_0+1)(M+1)}y_i''\geq2,\]by the pigeon hole principle. This shows that the word $y_{l_0(M+1)+1}''\dots y_{(l_0+1)(M+1)}''$ includes the consecutive $0$'s whose length is at most $M-1$, which contradicts $\sigma^n\uy''\preceq \ud$ for all $n\geq0$. This finishes the proof. 

    (3) Since $\be$ satisfies
    \[1=\Biggl(\frac{1}{\be}+\dots+\frac{1}{\be^N}+\frac{0}{\be^{N+1}}\Biggr)\Biggl(1+\frac{1}{\be^{N+1}}+\frac{1}{\be^{2(N+1)}}+\dots\Biggr),\]
    we have that $\ud=\overline{\overline{1}^{N}0}^\infty$. By Item (2), we know that $(N-1)/N\leq c_\be \leq N/(N+1)$. Since $\ud$ is periodic, the periodic measure $\displaystyle \hat\mu:=\frac{1}{N+1}\sum_{i=0}^N\delta_{\sigma^i\ud}$
    is $\sigma$-invariant and satisfies $\int_{X_\be}\chi_{C_1} d\hat\mu=N/(N+1)$, which ensures that $c_\be=N/(N+1)$. 

    (4) Since $\be$ satisfies
    \[1=\Biggl(\frac{1}{\be}+\frac{0}{\be^2}+\frac{0}{\be^{M+1}}\Biggr)\Biggl(1+\frac{1}{\be^{M+1}}+\frac{1}{\be^{2(M+1)}}+\dots\Biggr),\]
    we have that $\ud=\overline{1\overline{0}^{M}}^\infty$. By Item (2), we know that $1/(M+2)\leq c_\be \leq 1/(M+1)$. Since $\ud$ is periodic, the periodic measure $\displaystyle \bar\mu:=\frac{1}{M+1}\sum_{i=0}^M\delta_{\sigma^i\ud}$
    is $\sigma$-invariant and satisfies $\int_{X_\be}\chi_{C_1} d\bar\mu=1/(M+1)$, which ensures that $c_\be=1/(M+1)$. 
\end{proof}

\begin{remark}
    The question of whether the function $\be\mapsto c_\be$ is continuous or not remains open. An affirmative answer would imply the existence of uncountably many $\be\in(1,2)$ for which an invariant probability measure attaining the supremum $\displaystyle c_\be=\sup_{\mu\in M(X_\be)}\int_{X_\be}\chi_{C_1} d\mu$ is non-periodic, i.e., it is not represented as a finite convex sum of periodic measures. This is in contrast to the typical situation for topologically mixing subshifts of finite type (see,  e.g., \cite{Je}). In fact, for $\ux\in X_\be$, the value $\displaystyle{\frac{1}{n}\sum_{i=0}^{n-1}\chi_{C_1}(\sigma^i\ux)}$ is rational, which yields that if the supremum is attained by a periodic measure then $c_\be$ is rational. This shows that if $c_\be$ is irrational, any probability measure attaining the supremum must be non-periodic.
\end{remark}


\end{section}

\begin{section}{A generalization of the Hata-{Yamaguti} formula}

In this section, we show a version of the Hata-{Yamaguti} formula (\cite[Theorem 4.6]{Ha-Ya}) for beta-expansions in the case of $1<\be<2$.

Let $1<\be\leq 2$ and $t\in\R$. Denote by $\nu_t$ the probability  measure on $X_\be$ as in Theorem \ref{Main C}(2).
Set $D_t(x)=\nu_t(\pi^{-1}_\be([0,x]))$ for $x\in[0,1]$, where $\pi_\be:X_\be\to [0,1]$ is the surjective map defined by 
\[\pi_\be((a_i)_{i=1}^\infty)=\sum_{i=1}^\infty\frac{a_i}{\be^i}\]
for $(a_i)_{i=1}^\infty\in X_\be$. By Proposition \ref{3-4} (1) and Theorem \ref{Main C} (1), we obtain
\[D_t(x)=
    \sum_{n=1}^\infty \frac{g_n(x)L_{n-1}(t,x)}{\la^n_t}
\]
for $x\in[0,1]$, where $\la_t$ is a unique leading eigenvalue of $\cL_t$ and 
\[L_n(t,x)=
H_n(t,(g_i(x))_{i=1}^\infty)=\begin{cases}
    1 & \text{ if } n=0 \\
    \exp\Biggl(\sum_{i=1}^n g_i(x)\cdot t\Biggr)& \text{ if } n\geq1
\end{cases}\]
for $x\in[0,1]$ and $n\geq0$. 
For $p \in(0,1)$ let $f(p)\in\R$ be the inverse function of $1/\la_t$ for $t\in\R$, i.e., the function satisfying $p=1/\la_{f(p)}$ for $p\in(0,1)$. We note that it is well-defined since $\la_t$ is strictly increasing and continuous by Proposition \ref{5-1} (1) and (2). The inverse function theorem and Proposition \ref{5-1} (2) yield that it can be extended to an analytic function on some region $\C$ including $(0,1)$. In the following, we set $F_p(x)=D_{f(p)}(x)$ for $p\in(0,1)$ and $x\in[0,1]$.
Since $p=1/\la_{f(p)}$ for $p\in(0,1)$, we have
\begin{align*}
    F_p(x)&=\sum_{n=1}^\infty g_n(x)p^n L_{n-1}(f(p),x) \\
          &=\frac{1}{e^{f(p)}} \sum_{n=1}^\infty g_n(x)p^{n-\sum_{i=1}^n g_i(x)}(pe^{f(p)})^{\sum_{i=1}^ng_i(x)}
    \end{align*}
for $x\in[0,1]$. 
In the case of $\be=2$, we can see that the function $F_p(x)$ is equal to the Lebesgue singular function as follows. In this case,   
\[\phi_t(\la_t^{-1})
=\sum_{n=1}^\infty\frac{H_{n-1}(t,\overline{1}^{\infty})}{\la_t^n}
=\sum_{n=1}^\infty\frac{e^{(n-1)t}}{\la_t^n}=\frac{1}{\la_t-e^t}
\]
for $t\in\R$ by the definition of the power series $\phi_t$. Then the equation $\phi_t(\la_t^{-1})=1$ yields $\la_t=e^t+1$. Since $p=1/\la_{f(p)}$, the function $f(p)$ satisfies $e^{f(p)}=(1-p)/p$. 
Hence  
\[\begin{split}
F_p(x)
&=\frac{1}{e^{f(p)}} \sum_{n=1}^\infty g_n(x)p^{n-\sum_{i=1}^n g_i(x)}(pe^{f(p)})^{\sum_{i=1}^ng_i(x)} \\
&=\frac{p}{1-p}\sum_{n=1}^\infty g_n(x)p^{n-\sum_{i=1}^n g_i(x)} (1-p)^{\sum_{i=1}^n g_i(x)},
\end{split}
\]
which is the Lebesgue singular function (see e.g., \cite{Al-Ka, Ha-Ya}). The Hata-Yamaguch formula \cite[Theorem 4.6]{Ha-Ya} shows that for $\be=2$ the function $F_p(x)$ is differentiable as a function of $p\in(0,1)$ at each $x\in[0,1]$ and
\[\frac{1}{2}\frac{\partial F_p(x)}{\partial p}\Biggr|_{p=\frac{1}{2}}=\mathcal{T}(x),\]
where $\mathcal{T}(x)$ is the Takagi function
defined by 
\[\mathcal{T}(x)=\sum_{n=1}^\infty\frac{T^n(x)}{2^n}\] for $x\in[0,1]$. Here, $T(x)$ is the tent map $T(x)=1-|1-2x|$. The function $\mathcal{T}$ is well-known as a continuous but nowhere differentiable function on $[0,1]$. 
Our main result of this section is an analogue of this result to the beta-map $\tau_\be$ for $1<\be<2$. 

\begin{theorem}\label{Main E}
There is an open neighborhood $U\subset (0,1)$ of $1/\be$ such that for each $x\in[0,1]$ the map $p\in U\mapsto F_p(x)\in[0,1]$ is real-analytic. Set  
    \[G(x)=\frac{1}{\be}\frac{\partial F_p(x)}{\partial p}\Biggr|_{p=\frac{1}{\be}}\]
for $x\in[0,1]$. Then the function $G$ is continuous but nowhere differentiable on $[0,1]$.    
\end{theorem}
For the proof of Theorem \ref{Main E}, we need the following two lemmas. 

\begin{lemma}\label{6-2} There is an open neighborhood $U\subset (0,1)$ of $1/\be$ such that for each $x\in[0,1]$ the function $p\in U\mapsto F_p(x)$ is real-analytic. 
The function 
    \[G(x)=\frac{1}{\be}\frac{\partial F_p(x)}{\partial p}\Biggr|_{p=\frac{1}{\be}}\]
    has the form 
    \begin{equation}\label{G(x)2}    
    G(x)=\frac{g_1(x)}{\be}+\sum_{n=2}^\infty\frac{g_n(x)}{\be^{n}}\Biggl(n-\frac{\sum_{i=1}^{n-1}g_i(x)}{m_{\be}([1/\be,1])}\Biggr)
     \end{equation}
     for $x\in[0,1]$, where $m_\be$ is a unique $\tau_\be$-invariant probability measure absolutely continuous with respect to the Lebesgue measure.
\end{lemma}

\begin{proof}
By Proposition \ref{5-1} (1) and (2), we can take an open set $V\subset \C$ including $(0,1)$ on which the extension of $f(p)$ is analytic. As a slight abuse of the notation, we also denote by $f$ its extension. We set $F_{z,N}(x)=\sum_{n=1}^N g_n(x)z^n L_{n-1}(f(z),x)$ for $z\in V$, $N\geq1$ and $x\in[0,1]$. Note that 
\begin{align*}
    \Biggl|\sum_{n=1}^N g_n(x)z^n L_{n-1}(f(z), x)\Biggr|\leq \sum_{n=1}^N|ze^{f(z)}|^n.
\end{align*}
By setting $W=\{z\in\C; |ze^{f(z)}|<1\}$, we have 
\[\lim_{N\to\infty} F_{z,N}(x)=F_z(x):=\sum_{n=1}^\infty g_n(x)z^n L_{n-1}(f(z),x)\]
uniformly on any compact set in $V\cap W$. Since $F_{z,N}(x)$ is analytic on $V\cap W$, the function $F_z(x)$ is an analytic function for $z$. 
Note that $|e^{f(1/\be)}/\be|=1/\be<1$, which yields $1/\be\in V\cap W$. By considering the function $z\in U:=V\cap W\cap(0,1)\mapsto F_z(x)\in\C$, we have that $p\in U\to F_p(x)\in\R$ is real-analytic. 

Let $p\in U$. Then 
\[
\begin{split}
&\Biggl|\frac{\partial}{\partial p}\Bigl(g_n(x)p^n L_{n-1}(f(p),x)\Bigr)\Biggr| \\
&=\Biggl|n g_n(x)p^{n-1} L_{n-1}(f(p),x)+g_n(x)p^n f'(p)\Bigl(\sum_{i=1}^{n-1}g_i(x)\Bigr)L_{n-1}(f(p),x)\Biggr| \\
&\leq n p^{n-1} e^{|f(p)|n}+n|f'(p)||p^n|e^{|f(p)|n}
\end{split}\]
for $n\geq2$ and 
\[\frac{\partial}{\partial p}\Bigl(g_1(x)p \cdot L_{0}(f(p),x)\Bigr)=g_1(x),\]
which yields
\[\begin{split}
\sum_{n=1}^\infty\Biggl|\frac{\partial}{\partial p}(g_n(x)p^n L_{n-1}(f(p),x))\Biggr| 
\leq \sum_{n=1}^\infty
(n p^{n-1} e^{|f(p)|n}+n|f'(p)||p^n|e^{|f(p)|n})<\infty.
\end{split}\]
By the Lebesgue convergence theorem, we obtain
\[\begin{split}
&\frac{\partial}{\partial p}\sum_{n=1}^\infty g_n(x)p^n L_{n-1}(f(p), x) \\
&=\sum_{n=1}^\infty\frac{\partial}{\partial p}(g_n(x)p^n L_{n-1}(f(p), x)) \\
&=\sum_{n=1}^\infty ng_n(x)p^{n-1} L_{n-1}(f(p), x)+\sum_{n=2}^\infty g_n(x)p^n f'(p)\Bigl(\sum_{i=1}^{n-1}g_n(x)\Bigr)L_{n-1}(f(p),x).
\end{split}
\]
If $p=1/\be$ we have that $f(1/\be)=0$ since $\la_{f(1/\be)}=\be$ and $\la_0=\be$. 
This gives $L_n(f(1/\be),x)=1$ for any $n\geq0$ and $x\in[0,1]$. Hence 
\begin{align}\label{G(x)}
    G(x)&=\frac{1}{\be}\frac{\partial F_p(x)}{\partial p}\biggl|_{p=\frac{1}{\be}} \notag \\ 
    &=\sum_{n=1}^\infty\frac{ng_n(x)}{\be^{n}}
    +\sum_{n=2}^\infty\frac{g_n(x)}{\be^{n+1}}f'\Bigl(\frac{1}{\be}\Bigr)\sum_{i=1}^{n-1}g_i(x) \notag \\
    &=\frac{g_1(x)}{\be}+\sum_{n=2}^\infty\frac{g_n(x)}{\be^{n}}\Biggl(n+\frac{f'(1/\be)}{\be}\sum_{i=1}^{n-1}g_i(x)\Biggr)
\end{align}
for $x\in[0,1]$. Since $1/\la_{f(p)}=p$ for $p\in(0,1)$, we have $\la'_{f(p)}f'(p)=-\la_{f(p)}/p$, which yields 
\[f'(p)=-\frac{\la_{f(p)}}{p\la'_{f(p)}}=-\frac{1}{p\al(f(p))}\]
for $p\in(0,1)$, where $\al$ is the function as defined in Proposition \ref{5-1} (3). Hence 
\[f'\Bigl(\frac{1}{\be}\Bigr)=-\frac{\be}{\al(0)}=-\frac{\be}{\mu_0(C_1)}=-\frac{\be}{m_\be([1/\be,1])}.\]
We recall that $\mu_t$ denotes the unique equilibrium state for the potential $t\chi_{C_1}$, where $t\in\R$. This equals to the measure of maximal entropy in the case of $t=0$, which yields $\mu_0(C_1)=m_\be([1/\be,1])$. Together with the equation (\ref{G(x)}), we have
\begin{equation*}
G(x)=\frac{g_1(x)}{\be}+\sum_{n=2}^\infty\frac{g_n(x)}{\be^{n}}\Biggl(n-\frac{\sum_{i=1}^{n-1}g_i(x)}{m_{\be}([1/\be,1])}\Biggr)
\end{equation*}
for $x\in[0,1]$, which ends the proof.
\end{proof}

\begin{lemma}\label{6-3}
Let $G(x)$ be the function defined as in Lemma \ref{6-2}. Then we have the following:

    (1) $G(0)=G(1)=1$.

    (2) For $x\in(0,1]$ set 
    \[H(x)=\frac{q_1(x)}{\be}+\sum_{n=2}^\infty\frac{q_n(x)}{\be^{n}}\Biggl(n-\frac{\sum_{i=1}^{n-1}q_i(x)}{m_{\be}([1/\be,1])}\Biggr),\]
    where $(q_n(x))_{n=1}^\infty$ is the coefficient sequence of the quasi-greedy expansion of $x\in(0,1]$. Then $H(x)=G(x)$ for $x\in(0,1]$.
    
\end{lemma}

\begin{proof}
    (1) By the formula (\ref{G(x)2}) it is obvious that $G(0)=0$. Since 
    \[G(1)=\frac{1}{\be}\frac{\partial F_p(1)}{\partial p}\Biggr|_{p=\frac{1}{\be}}\]
    and $F_p(1)=D_{f(p)}(1)=\nu_{f(p)}(\pi_\be^{-1}[0,1])=1$ for any $p\in(0,1)$, we obtain 
    \[G(1)=\frac{1}{\be}\frac{\partial 1}{\partial p}\Biggr|_{p=\frac{1}{\be}}=0.\]
    
    (2) Let $x\in(0,1]$ be non-simple, i.e., there is no positive integer $n_0$ such that $\tau_\be^{n_0-1}(1)=1/\be$. In this case, we have $(q_n(x))_{n=1}^\infty=(g_n(x))_{n=1}^\infty$, which yields $H(x)=G(x)$. 

    Let $x\in(0,1]$ be simple, i.e., there is a positive integer $n_0\geq1$ such that $\tau_\be^{n_0-1}(1)=1/\be$. In this case, we know that $g_n(x)=0$ for $n\geq n_0+1$. By the definition of the quasi-greedy expansion of $x$ we have $q_i(x)=g_i(x)$ for $1\leq i<n_0$, $q_{n_0}(x)=0$ and $q_{i+n_0}(x)=q_{i}(1)$ for $i\geq1$. Set 
    \[S_ng(x)=\begin{cases}
        0, & (n=0)\\
        \sum_{i=1}^n g_i(x), & (n\geq1)
    \end{cases}\]
    and 
    \[S_nq(x)=\begin{cases}
        0, & (n=0)\\
        \sum_{i=1}^n q_i(x). & (n\geq1)
    \end{cases}\]    
    Then 
    \begin{align*}
        H(x)&=\sum_{n=1}^\infty\frac{q_n(x)}{\be^n}\Biggl(n-\frac{S_{n-1}q(x)}{m_{\be}([1/\be,1])}\Biggr) \\
        &=\sum_{n=1}^{n_0}\frac{q_n(x)}{\be^n}\Biggl(n-\frac{S_{n-1}q(x)}{m_{\be}([1/\be,1])}\Biggr)+\sum_{n=n_0+1}^\infty\frac{q_n(x)}{\be^n}\Biggl(n-\frac{S_{n-1}q(x)}{m_{\be}([1/\be,1])}\Biggr) \\
        &=\sum_{n=1}^{n_0}\frac{g_n(x)}{\be^n}\Biggl(n-\frac{S_{n-1}g(x)}{m_{\be}([1/\be,1])}\Biggr)-\frac{1}{\be^{n_0}}\Biggl(n_0-\frac{S_{n_0-1}g(x)}{m_{\be}([1/\be,1])}\Biggr)\\
        &+\sum_{n=1}^\infty\frac{q_n(1)}{\be^{n_0+n}}\Biggl(n_0+n-\frac{S_{n_0+n-1}q(x)}{m_{\be}([1/\be,1])}\Biggr) \\
        &=G(x)-\frac{n_0}{\be^{n_0}}+\frac{1}{\be^{n_0}}\frac{S_{n_0-1}g(x)}{m_\be([1/\be,1])} \\
        &+\frac{1}{\be^{n_0}}\sum_{n=1}^\infty\frac{q_n(1)}{\be^n}\Biggl(n+n_0-\frac{S_{n_0-1}g(x)+S_{n-1}q(1)}{m_\be([1/\be,1])}\Biggr) \\
        &=G(x)-\frac{n_0}{\be^{n_0}}\Biggl(1-\sum_{n=1}^\infty\frac{q_n(1)}{\be^n}\Biggr)+\frac{1}{\be^{n_0}}\Biggl(1-\sum_{n=1}^\infty\frac{q_n(1)}{\be^n}\Biggr)\frac{S_{n_0-1}g(x)}{m_\be([1/\be,1])} \\
        &+\frac{1}{\be^{n_0}}\sum_{n=1}^\infty\frac{q_n(1)}{\be^n}\Biggl(n-\frac{S_{n-1}q(1)}{m_\be([1/\be,1])}\Biggr) \\
        &=G(x)+\frac{1}{\be^{n_0}}H(1).
        \end{align*}
        If $1$ is non-simple, we have $H(x)=G(x)$ since $H(1)=G(1)=0$. If $1$ is simple, we obtain $H(1)=G(1)+(1/\be^{m_0})H(1)=(1/\be^{m_0})H(1)$ by the above equation, where $m_0$ is the minimal positive integer satisfying $\tau^{m_0-1}(1)=1/\be$. This shows that $H(1)=0$, which yields $H(x)=G(x)+(1/\be^{m_0})H(1)=G(x)$, as desired. 
    \end{proof}

\begin{proof}[Proof of Theorem \ref{Main E}]
Let $N\geq2$ be an integer. For $x\in[0,1)$ set 
\[G_N(x)=\frac{g_1(x)}{\be}+\sum_{n=2}^N\frac{g_n(x)}{\be^n}\Biggl(n-\frac{\sum_{i=1}^{n-1}g_i(x)}{m_\be([1/\be,1])}\Biggr).\]
Since $g_n(x)$ is right continuous for any $1\leq n\leq N$ so is $G_N(x)$. Hence 
\[\Bigl|G_N(x)-G(x)\Bigr|\leq \sum_{n=N+1}^\infty \frac{n}{\be^n}\Bigl(1+\frac{1}{m_\be([1/\be,1])}\Bigr)\to 0\]
as $N\to\infty$, which yields that $G_N(x)$ converges to $G(x)$ uniformly on $x\in[0,1)$ as $N\to0$. This shows that $G(x)$ is right continuous for $x\in[0,1)$. 

Let $N\geq2$ be an integer. For $x\in(0,1]$ set 
\[H_N(x)=\frac{q_1(x)}{\be}+\sum_{n=2}^N\frac{q_n(x)}{\be^n}\Biggl(n-\frac{\sum_{i=1}^{n-1}q_i(x)}{m_\be([1/\be,1])}\Biggr).\]
Since $q_n(x)$ is left continuous for any $1\leq n\leq N$ so is $H_N(x)$. As in the same way of the case of the functions $G_N(x)$ and $G(x)$, we can show that $H_N(x)$ converges to $H(x)$ uniformly on $(0,1]$. Since $G(x)=H(x)$ for $x\in(0,1]$ by Lemma \ref{6-3} (2), we have that $G(x)$ is left continuous on $(0,1]$. Together with the right-continuity of $G(x)$ on $[0,1)$, we have that $G(x)$ is continuous on $[0,1]$. 

We show that $G(x)$ is nowhere differentiable on $[0,1]$. Let $x\in(0,1]$. By Lemma \ref{6-3} (2) we know that 
\[G(x)=H(x)=\frac{q_1(x)}{\be}+\sum_{n=2}^\infty\frac{q_n(x)}{\be^{n}}\Biggl(n-\frac{\sum_{i=1}^{n-1}q_i(x)}{m_{\be}([1/\be,1])}\Biggr).\]
Note that there are infinitely many $n$ such that $q_n(x)=1$. Set $l(1)=1$ and $l(N)=\min\{n>l(N-1); q_n(x)=1\}$ for $N\geq2$. Let 
\[x_N=\sum_{n=1}^{l(N)}\frac{q_n(x)}{\be^n}\]
and 
\[x_N^{-}=\sum_{n=1}^{l(N-1)}\frac{q_n(x)}{\be^n}+\frac{1}{\be^{l(N)+1}}\]
for $N\geq2$. 

Now we assume that $G(x)$ is differentiable at $x\in(0,1]$ and set $\alpha=\lim_{y\to x}(G(x)-G(y))/(x-y)$. Then 
\begin{align*}
    &\Biggl|\frac{G(x_{N+1})-G(x_{N})}{x_{N+1}-x_N}-\alpha\Biggr| \\
    &=\Biggl|\frac{G(x_{N+1})-G(x_{})}{x_{N+1}-x}\cdot\frac{x_{N+1}-x}{x_{N+1}-x_N}+\frac{G(x_{})-G(x_{N})}{x_{}-x_N}\cdot\frac{x_{}-x_N}{x_{N+1}-x_N}-\alpha\Biggr|     \\
    &\leq \Biggl|\Bigl(\frac{G(x_{N+1})-G(x_{})}{x_{N+1}-x}-\alpha\Bigr)\Biggr|\Biggl|\frac{x_{N+1}-x}{x_{N+1}-x_N}\Biggr|+\Biggl|\Bigl(\frac{G(x_{})-G(x_{N})}{x_{}-x_N}-\alpha\Bigr)\Biggr|\Biggl|\frac{x_{}-x_N}{x_{N+1}-x_N}\Biggr|     \\
    &\leq \Biggl|\Bigl(\frac{G(x_{N+1})-G(x_{})}{x_{N+1}-x}-\alpha\Bigr)\Biggr|\Biggl|\frac{1}{\beta-1}\Biggr|+\Biggl|\Bigl(\frac{G(x_{})-G(x_{N})}{x_{}-x_N}-\alpha\Bigr)\Biggr|\Biggl|\frac{\be}{\be-1}\Biggr| \to 0     
    \end{align*}
as $N\to\infty$. In the same way, we obtain 
\begin{align*}
    &\Biggl|\frac{G(x_{N+1}^-)-G(x_{N})}{x_{N+1}^--x_N}-\alpha\Biggr| \\
    &\leq \Biggl|\Bigl(\frac{G(x_{N+1}^-)-G(x_{})}{x_{N+1}^--x}-\alpha\Bigr)\Biggr|\Biggl|\frac{x_{N+1}^--x}{x_{N+1}^--x_N}\Biggr|+\Biggl|\Bigl(\frac{G(x_{})-G(x_{N})}{x_{}-x_N}-\alpha\Bigr)\Biggr|\Biggl|\frac{x_{}-x_N}{x_{N+1}^--x_N}\Biggr|     \\
    &\leq \Biggl|\Bigl(\frac{G(x_{N+1})-G(x_{})}{x_{N+1}-x}-\alpha\Bigr)\Biggr|\Biggl|\frac{\be^2}{\beta-1}-\be\Biggr|+\Biggl|\Bigl(\frac{G(x_{})-G(x_{N})}{x_{}-x_N}-\alpha\Bigr)\Biggr|\Biggl|\frac{\be^2}{\be-1}\Biggr| \to 0     
    \end{align*}
as $N\to\infty$. By the formula (\ref{G(x)2}) of $G(x)$, however, we have 
\begin{align*}
    \frac{G(x_{N+1})-G(x_{N})}{x_{N+1}-x_N}&=\be^{l(N+1)}\Biggl(\frac{l(N+1)-\sum_{i=1}^{l(N+1)-1}g_i(x_{N+1})/m_\be([1/\be,1])}{\be^{l(N+1)}}\Biggr) \\
    &=l(N+1)-\frac{N}{m_\be([1/\be,1])}
\end{align*}
and 
\begin{align*}
    \frac{G(x_{N+1}^-)-G(x_{N})}{x_{N+1}^--x_N}&=\be^{l(N+1)+1}\Biggl(\frac{l(N+1)+1-\sum_{i=1}^{l(N+1)}g_i(x_{N+1}^{-})/m_\be([1/\be,1])}{\be^{l(N+1)+1}}\Biggr) \\
    &=l(N+1)+1-\frac{N}{m_\be([1/\be,1])}.
\end{align*}
This shows that 
\[\frac{G(x_{N+1}^-)-G(x_{N})}{x_{N+1}^--x_N}=\frac{G(x_{N+1})-G(x_{N})}{x_{N+1}-x_N}+1.\]
By taking the limit of both sides of the above equation as $N\to\infty$, we have $\alpha=\alpha+1$, which gives the contradiction. 

For $x=0$, we know that $G(0)=0$. Let $M\geq1$ and set $x_M=1/\be^M$. Then $G(x_M)=M/\be^M$ by the formula (\ref{G(x)2}) and 
\[\frac{G(x_M)-G(0)}{x_M-0}=\be^M\cdot\frac{M}{\be^M}=M\to\infty\]
    as $M\to\infty$, which shows that $G(x)$ is non-differentiable at $x=0$. This finishes the proof. 
\end{proof}
\end{section}

\begin{section}{Examples}
In this section, we show that $\dim_H\La_\al$ has a simple formula as a function of $\al$ if $\be>1$ is a positive solution of the equation $x^{N+1}-x^{N}-1=0$ for some positive integer $N$. 

Let $N\geq1$ be an integer and let $\be>1$ satisfy $\be^{N+1}-\be^N-1=0$. It is easy to see that the greedy expansion of $1$ is given by $1\overline{0}^{(N-1)}1\overline{0}^\infty$, which yields that the quasi-greedy expansion of $1$ is given by $\ud=\overline{1\overline{0}^N}^\infty$. 

\begin{theorem}\label{7-1}
    Let $\be>1$ satisfy $\be^{N+1}-\be^N-1=0$ for some positive integer $N$. Then for $\al\in(0,1/(N+1))$ we obtain 
    \[\dim_{H}(\Lambda_\al)=\frac{1}{\log\be}\Biggl((1-N\al)\log\frac{1-N\al}{1-(N+1)\al}-\al\log\frac{\al}{1-(N+1)\al}\Biggr).\]
\end{theorem}

\begin{proof}
   Let $t\in\R$ and let $\la_t\in(1,\infty)$ be the positive number as in Theorem \ref{Main D}. Since
\begin{align*}
    1=\phi_t(\la_t^{-1})
    &=\frac{1}{\la_t}\Biggl(1+\frac{e^t}{\la_t^{N+1}}+\frac{e^{2t}}{\la_t^{2(N+1)}}+\dots \Biggr) \\
    &=\frac{1}{\la_t}\frac{1}{1-e^t/\la_t^{N+1}} \\
    &=\frac{\la_t^N}{\la_t^{N+1}-e^t},
\end{align*}
   we obtain $\la_t^{N+1}-\la_t^N-e^t=0$. This shows that 
\begin{equation}\label{ex}
e^t=\la_t^{N+1}-\la^N_{t}.
\end{equation}
Since $\la_t$ is real-analytic for $t\in(-\infty,\infty)$ by Proposition \ref{5-1}(2), differentiating both sides of the above equation by $t$ gives
\[e^t=(N+1)\la_t^N \la_t'-N\la_{t}^{N-1}\la_t',\]
which yields 
\begin{align*}
    \frac{\la_t'}{\la_t}
    &=\frac{e^t}{(N+1)\la_t^{N+1}-N\la_t^N}=\frac{\la_t^{N+1}-\la_t^N}{(N+1)\la_t^{N+1}-N\la_t^N}    \\
    &=\frac{\la_t-1}{(N+1)\la_t-N}.
    \end{align*}
Let $\al\in(0,1/(N+1))$ and let $t(\al)\in\R$ be the real number such that $\al=\la'_{t(\al)}/\la_{t(\al)}$. Note that we can take such $t(\al)$ by Proposition \ref{5-1}(4) and (5).

Since
\[\al=\frac{\la'_{t(\al)}}{\la_{t(\al)}}=\frac{\la_{t(\al)}-1}{(N+1)\la_{t(\al)}-N},\]
we obtain
\[\la_{t(\al)}=\frac{1-N\al}{1-(N+1)\al}.\]
By the equation (\ref{ex}), we have
\[t(\al)=\log(\la_{t(\al)}^{N+1}-\la_{t(\al)}^{N})=N\log{\la_{t(\al)}}+\log(\la_{t(\al)}-1).\]
Applying Theorem \ref{Main D}, we obtain 
\begin{align*}
    \dim_{H}\Lambda_\al
    &=\frac{1}{\log\be}\Bigl(\log\la_{t(\al)}-t(\al)\al\Bigr) \\
    &=\frac{1}{\log\be}\Bigl(\log\la_{t(\al)}-\al N\log\la_{t(\al)}-\al\log(\la_{t(\al)}-1)\Bigr) \\
    &=\frac{1}{\log\be}\Biggl((1-\al N)\log \frac{1-N\al}{1-(N+1)\al}-\al\log\frac{\al}{1-(N+1)\al} \Biggr),
    \end{align*}
as desired.
\end{proof}

\end{section}

Acknowledgements. The author would like to thank Hiroki Sumi for his valuable comments about the Hata-{Yamaguti} formula. The author would also like to thank the anonymous referee for valuable comments and suggestions. This work was supported by JSPS KAKENHI Grant Number 20K14331 and 24K16932.

\bibliographystyle{amsplain}

\end{document}